\documentclass[reqno]{amsart}	

\usepackage{enumerate}
\usepackage{graphicx}
\usepackage{nicefrac}
\usepackage{booktabs}
\usepackage{amsaddr, amsmath, amssymb}
\usepackage{stmaryrd}
\usepackage{caption}
\usepackage{subcaption}
\usepackage{tikz}
\usepackage{pgfplots}
\usepackage{geometry}
\usepackage{color}
\usepackage[onehalfspacing]{setspace}
\allowdisplaybreaks

\pgfplotsset{compat=1.15}
\definecolor{darkbrown}{rgb}{0.57, 0.40, 0.13}

\newcommand{\myspace}[1]{\mathbb{#1}}
\newcommand{\myvec}[1]{\mathfrak{#1}}          % vectors
\newcommand{\mymatrix}[1]{\boldsymbol{#1}}     % matrices
\newcommand{\E}{\mathsf{E}}
\newcommand{\osc}{\operatorname{osc}}

\newtheorem{theorem}{Theorem}

\newtheorem{corollary}[theorem]  {Corollary}

\newtheorem{lemma}[theorem]      {Lemma}

\theoremstyle{definition}

\newtheorem*{remark*}	{Remark}

\begin{document}

\title[A Posteriori Error Estimates for $hp$-FE Discretizations in Elastoplasticity]{A Posteriori Error Estimates for $hp$-FE Discretizations in Elastoplasticity}

\author[P.~Bammer, L.~Banz \and A.~Schr\"{o}der]{Patrick Bammer, Lothar Banz \and Andreas Schr\"{o}der}

\address{Fachbereich Mathematik, Paris Lodron Universit\"{a}t Salzburg,\\ Hellbrunner~Str.~14, 5020 Salzburg, Austria}

\begin{abstract}
In this paper, a reliable a posteriori error estimator for a model problem of elastoplasticity with linear kinematic hardening is derived, which satisfies some (local) efficiency estimates. It is applicable to any discretization that is conforming with respect to the displacement field and the plastic strain. Furthermore, the paper presents $hp$-finite element discretizations relying on a variational inequality as well as on a mixed variational formulation and discusses their equivalence by using biorthogonal basis functions. Numerical experiments demonstrate the applicability of the theoretical findings and underline the potential of $h$- and $hp$-adaptive finite element discretizations for problems of elastoplasticity.
\end{abstract}


\thanks{%
A.~Schr\"{o}der acknowledges the support by the Bundesministerium f\"{u}r Bildung, Wissenschaft und Forschung (BMBWF) under the Sparkling Science project SPA 01-080 'MAJA -- Mathematische Algorithmen für Jedermann Analysiert'.
}


\maketitle
	
% ----------------------------------------------------------------------------------------
%   SECTION: Introduction
% ----------------------------------------------------------------------------------------

\section{Introduction}%\label{sec:Introduction}

Elastoplasticity with hardening appears in many problems of mechanical engineering. Thereby, the holonomic constitutive law represents a well established model for elastoplasticity with linear kinematic hardening, which allows for the incremental computation of the deformation of an elastoplastic body, see e.g.~\cite{han1991finite, Han2013}. A well known weak formulation of a (pseudo-)time step of this model takes the form of a variational inequalitiy of the second kind that includes a non-differentiable plasticity functional $\psi(\cdot)$. One possible way to resolve the non-differentiability of $\psi(\cdot)$ is to regularize it as proposed e.g.~in \cite{reddy1987variational}. Another possibility to avoid difficulties resulting from the non-differentiability of $\psi(\cdot)$ consists in the introduction of an appropriate Lagrange multiplier within a mixed formulation, see e.g.~\cite{han1991finite, han1995finite, schroder2011error}. This, in particular, offers discretization approaches for problems of elastoplasticity with finite elements. However, the use of a Lagrange multiplier as an additional variable leads to a substantial increase of the number of degrees of freedom in the discretization, as it contains the same number of unknowns as the plastic strain variable. 

A~posteriori error control is an essential tool to measure the quality of the discretization and to steer adaptive finite element schemes. Typically, it relies on
the derivation of upper and lower bounds by specifying reliable and efficent a posteriori error estimators \cite{Ainsworth2000,Verfuerth2013}. Error control approaches for finite elements of low order in the context of elastoplasticity with hardening can be found in \cite{Alberty1999,carstensen1999numerical,Carstensen2003,carstensen2006reliable,schroder2011error}. We refer to \cite{Carstensen2016} on the optimal convergence of adaptive schemes based on a~posteriori error control. 

In this paper, we present a residual-based a~posteriori error estimator for a model problem of elastoplasticity with linear kinematic hardening. The error estimator is derived from upper and lower error estimates based on a suitable auxiliary problem given by a variational equation. A similar concept is used for the derivation of error estimates in the context of contact problems, see e.g.~\cite{Braess2005,Vesser2001}, with respect to low-order finite elements an \cite{banz2022priori, banz2020posteriori, banz2015biorthogonal, Banz2021abstract, burg2015posteriori, Schroeder2009, Schroeder2010, Schroeder2012} for finite elements of higher-order. We prove the reliability of the error estimator and show that it satisfies some (local) efficiency estimates, which are, however, suboptimal for higher-order methods in terms of the plastic strain as expected. The proposed approach is applicable to any discretization that is conforming with respect to the displacement field and the plastic strain. In particular, it can be applied to finite elements of higher-order or even to $hp$-finite elements (with varying mesh sizes and local polynomial degrees) as well as to approximations of them (e.g.~resulting from iterative solution schemes). For this purpose, we discuss three closely related $hp$-finite element approaches which can be used to discretize the model problem of elastoplasticity. The first one relies on the discretization of the variational inequality, where the non-differentiable plasticity functional $\psi(\cdot)$ is approximated by interpolation, cf.~\cite{Bammer2022Icosahom,gwinner2009p}. The two further discretizations are based on a mixed variational formulation and only differ in the choice of the set of admissible discrete Lagrange multipliers. Under a rather weak assumption on the shapes of the mesh elements all three discretizations turn out to be equivalent. This equivalence can be shown with the help of biorthogonal basis functions, which in addition allow to decouple the constraints associated with the discrete Lagrange multiplier, see \cite{Bammer2022Icosahom,banz2015biorthogonal}. To illustrate the applicability of the discretization approaches we consider several numerical examples. In particular, we discuss $h$- and $hp$-adaptive schemes steered by the proposed residual based a~posteriori error estimator. We observe that the convergence rates of the adaptive refinements are significantly superior to those of the uniform refinements as the associated finite element spaces are adapted to the singular behavior of the solution. This particularly concerns the free boundary resulting from the transition from pure elastic to elastoplastic deformation. According to these observations, the experiments demonstrates the potential of $h$- and $hp$-adaptivity for problems of elastoplasticity. \\

The paper is structured as follows: In Section~\ref{sec:model} the model problem of elastoplasticity with linear kinematic hardening and its weak formulation as a variational inequality of the second kind as well as a mixed variational formulation are presented. In Section~\ref{sec:discretization} the three $hp$-finite element discretizations (one based on the variational inequality and two based on the mixed variational formulation) are introduced. The upper and lower error estimates (relying on a suitable auxiliary problem) and the derivation of the posteriori error estimator are discussed in Section~\ref{sec:abstractAposteriori}. Finally, several numerical experiments highlighting the applicability of the theoretical findings can be found in Section~\ref{sec:numeric}. 

\medskip

% ---------------------------------------------------------------------------------------
%   SECTION: Elastopl. with linear kinematic hardening and its weak formulation
% ----------------------------------------------------------------------------------------

\section{Elastoplasticity with Linear Kinematic Hardening}
\label{sec:model}

Let $\Omega\subset \myspace{R}^d$ with $d\in\lbrace 2,3\rbrace$ be a bounded, polygonal domain with Lipschitz-boundary $\Gamma:=\partial\Omega$ and outer unit normal $\myvec{n}$ and define the vector spaces
\begin{align*}
 V := \Big\lbrace \myvec{v} \in H^1(\Omega, \myspace{R}^d) \; ; \; \myvec{v}_{\, | \, \Gamma_D} = \myvec{o} \Big\rbrace, \qquad
 Q := L^2(\Omega,\myspace{S}_{d,0}),
\end{align*}
where
\begin{align*}
 \myspace{S}_{d,0} := \bigg\lbrace \mymatrix{\tau} \in \myspace{R}^{d \times d}  \; ; \; \mymatrix{\tau} = \mymatrix{\tau}^{\top},\  \operatorname{tr}(\mymatrix{\tau}) := \sum_{i=1}^d \tau_{ii} = 0 \bigg\rbrace.
\end{align*}
Moreover, let $( \cdot,\cdot)_{0,\Omega}$ denote the $L^2(\Omega, X)$ inner product for $X\in\lbrace \myspace{R}, \myspace{R}^d, \myspace{R}^{d \times d}\rbrace$, inducing the $L^2$-norm $ \Vert \cdot \Vert_{0, \Omega}$. Due to Korn's inequality we may equip $V$ with the norm $\Vert \myvec{v}\Vert_{1,\Omega} := \big( \Vert \myvec{v} \Vert_{0,\Omega}^2 + |\myvec{v}|_{1,\Omega}^2 \big)^{1/2}$, where $|\myvec{v}|_{1,\Omega}^2 := \big(\mymatrix{\varepsilon}(\myvec{v}),\mymatrix{\varepsilon}(\myvec{v})\big)_{0,\Omega}$ with the linearized strain tensor $\mymatrix{\varepsilon}(\myvec{v}) := \frac{1}{2} \, \big(\nabla\myvec{v} + (\nabla\myvec{v})^{\top} \big)$. Furthermore, let $V^*$ be the dual space of $V$ equipped with the dual norm $\Vert \cdot \Vert_{V^*}$. We use the expression $A \lesssim B$ to hide the constant $c>0$ in the expression $A\leq c \, B$ if $c$ is independent of the element size $h$ and the polynomial degree $p$. If we wish to emphasize that the hidden constant $c$ depends on $x$ we may write $\lesssim_x$. If $A\lesssim B$ and $B\lesssim A$ we use the notation $A\approx B$.

%\pagebreak

% ----------------------------------------------------------------------------------------
\subsection{The model problem}

The model problem of elastoplasticity with linear kinematic hardening, see e.g.~\cite{Bammer2023Apriori,Han2013}, is to find a displacement field $\myvec{u} \in V$ and a plastic strain $\mymatrix{p} \in Q$ such that for a given volume force $\myvec{f}$ and a given surface force $\myvec{g}$ there holds
\begin{subequations}\label{eq:model_problem}
\begin{alignat}{2}
 -\operatorname{div} \mymatrix{\sigma}(\myvec{u},\mymatrix{p}) &= \myvec{f}   & \quad & \text{in } \Omega, \label{eq:model_PDE} \\
 \myvec{u} &= \myvec{o}                                            & \quad & \text{on } \Gamma_D, \\
 \mymatrix{\sigma}(\myvec{u},\mymatrix{p}) \, \myvec{n} &= \myvec{g}  & \quad & \text{on } \Gamma_N, \label{eq:model_boundaryCond} \\
 \mymatrix{\sigma}(\myvec{u},\mymatrix{p})-\myspace{H} \, \mymatrix{p} &\in \partial j(\mymatrix{p}) & \quad &  \text{in } \Omega .\label{eq:differntial_inclusion}
\end{alignat}
\end{subequations}
Thereby, the Dirichlet boundary part $\Gamma_D\subseteq\Gamma$ is assumed to be closed and to have positive surface measure. The Neumann boundary part is given by $\Gamma_N:=\Gamma\setminus\Gamma_D$. Moreover, the stress tensor is defined as $\mymatrix{\sigma}(\myvec{u},\mymatrix{p}) := \myspace{C} \, (\mymatrix{\varepsilon}(\myvec{u})-\mymatrix{p})$. Here, the elasticity tensor $\myspace{C}$ as well as the hardening tensor $\myspace{H}$ are assumed to be symmetric, uniformly bounded and uniformly elliptic, see e.g.~\cite{Bammer2023Apriori}. Furthermore, $\partial j(\cdot)$ represents the subdifferential of the plastic dissipation functional $j(\cdot)$ which is given by $j(\mymatrix{q}) := \sigma_y \, \vert\mymatrix{q}\vert_F$, where $\vert \cdot \vert_F$ denotes the Frobenius norm induced by the Frobenius inner product $\mymatrix{p}:\mymatrix{q} := \sum_{i,j} \mymatrix{p}_{ij} \, \mymatrix{q}_{ij}$. For the ease of presentation, the yield stress $\sigma_y>0$ in uniaxial tension is assumed to be constant. Recall that the deviatoric part of a matrix $\mymatrix{\tau}\in\myspace{R}^{d\times d}$ is given by $\operatorname{dev}(\mymatrix{\tau}) := \mymatrix{\tau} - \frac{1}{d} \, \operatorname{tr}(\mymatrix{\tau}) \, \mymatrix{I}$ where $\mymatrix{I} \in \myspace{R}^{d\times d}$ is the identity matrix.
Since $\operatorname{tr}(\mymatrix{q}) = 0$ for any $\mymatrix{q}\in Q$ we immediately obtain the identity
\begin{align}\label{eq:identity_deviatoricPart}
 \big( \operatorname{dev}(\mymatrix{\mu}), \mymatrix{q} \big)_{0,\Omega}
 = ( \mymatrix{\mu}, \mymatrix{q})_{0,\Omega} - \frac{1}{d} \, \big( \operatorname{tr}(\mymatrix{\mu}) \, \mymatrix{I}, \mymatrix{q} \big)_{0,\Omega}
 = ( \mymatrix{\mu}, \mymatrix{q})_{0,\Omega} - \frac{1}{d} \, \big( \operatorname{tr}(\mymatrix{\mu}), \operatorname{tr}(\mymatrix{q}) \big)_{0,\Omega}
 = ( \mymatrix{\mu}, \mymatrix{q})_{0,\Omega} 
 % \qquad \forall \, \mymatrix{q}\in Q,\ \mymatrix{\mu}\in L^2(\Omega,\myspace{R}^{d\times d}).
\end{align}
for all $\mymatrix{q}\in Q$ and $\mymatrix{\mu}\in L^2(\Omega,\myspace{R}^{d\times d})$.

% ----------------------------------------------------------------------------------------
\subsection{Weak formulations}

A well established weak formulation of \eqref{eq:model_problem} is given by the following variational inequality of the second kind: Find a pair $(\myvec{u},\mymatrix{p})\in V\times Q$ such that
\begin{align}\label{eq:variq_second_kind}
 a\big( (\myvec{u},\mymatrix{p}), (\myvec{v}-\myvec{u},\mymatrix{q}-\mymatrix{p}) \big) + \psi(\mymatrix{q}) - \psi(\mymatrix{p}) \geq \ell(\myvec{v}-\myvec{u}) \qquad 
 \forall \, (\myvec{v},\mymatrix{q})\in V\times Q.
\end{align}
Thereby, the bilinear form 
\begin{align}\label{eq:defi_bilinearF_a}
 a\big( (\myvec{u},\mymatrix{p}),(\myvec{v},\mymatrix{q}) \big) 
  := \big( \mymatrix{\sigma}(\myvec{u},\mymatrix{p}), \mymatrix{\varepsilon}(\myvec{v})-\mymatrix{q} \big)_{0,\Omega} + (\myspace{H} \, \mymatrix{p}, \mymatrix{q})_{0,\Omega}
\end{align}
for $(\myvec{u},\mymatrix{p}), (\myvec{v},\mymatrix{q})\in V\times Q$ is continuous and $(V\times Q)$-elliptic, i.e.~there exist constants $c_a,\alpha>0$ such that
\begin{align*}
 a\big( (\myvec{u},\mymatrix{p}),(\myvec{v},\mymatrix{q}) \big) 
 \leq c_a \, \Vert (\myvec{u},\mymatrix{p})\Vert \, \Vert (\myvec{v},\mymatrix{q}) \Vert, \quad
 \alpha \, \Vert (\myvec{v},\mymatrix{q})\Vert^2 
 \leq a\big( (\myvec{v},\mymatrix{q}),(\myvec{v},\mymatrix{q}) \big) \qquad 
 \forall \, (\myvec{u},\mymatrix{p}), (\myvec{v},\mymatrix{q})\in V\times Q,
\end{align*}
respectively, where $\Vert (\myvec{v},\mymatrix{q}) \Vert := \big( \Vert \myvec{v}\Vert_{1,\Omega}^2 + \Vert \mymatrix{q}\Vert_{0,\Omega}^2 \big)^{1/2}$ is a norm on the Hilbert space $V \times Q$, see e.g.~\cite{Han2013}.
The convex, continuous and subdifferential plasticity functional $\psi(\cdot)$ and the continuous load functional $\ell(\cdot)$ in \eqref{eq:variq_second_kind} are given by
\begin{align*}
 \psi(\mymatrix{q}) := (\sigma_y,\vert \mymatrix{q}\vert_F)_{0,\Omega}, \quad
 \ell(\myvec{v}) := \langle \myvec{f},\myvec{v} \rangle + \langle \myvec{g},\myvec{v} \rangle_{\Gamma_N}
\end{align*}
for $(\myvec{v},\mymatrix{q})\in V\times Q$, where $\langle \cdot, \cdot \rangle$ and $\langle \cdot , \cdot \rangle_{\Gamma_N}$ denote the duality pairing between $V^*$ and $V$, and $H^{-1/2}(\Gamma_N,\myspace{R}^d)$ and $\widetilde{H}^{1/2}(\Gamma_N,\myspace{R}^d):= V|_{\Gamma_N}$, respectively.
In many cases it is beneficial to consider a mixed variational formulation of the variational inequality \eqref{eq:variq_second_kind}. For this purpose, let
\begin{align*}
 \Lambda 
 := \big\lbrace \mymatrix{\mu}\in Q \; ; \; \vert\mymatrix{\mu}\vert_F \leq \sigma_y \text{ a.e.~in }\Omega \big\rbrace
\end{align*}
be the non-empty, convex and closed set of admissible Lagrange multipliers. It is shown in \cite{Bammer2023Apriori} that $\Lambda$ can be alternatively represented as
\begin{align*}
 \Lambda
 = \big\lbrace \mymatrix{\mu}\in Q \; ; \; (\mymatrix{\mu}, \mymatrix{q})_{0,\Omega} \leq \psi(\mymatrix{q}) \text{ for all } \mymatrix{q}\in Q \big\rbrace.
\end{align*}
A mixed variational formulation is to find a triple $(\myvec{u},\mymatrix{p},\mymatrix{\lambda})\in V\times Q\times \Lambda$ such that
\begin{subequations}\label{eq:mixed_variationalF}
\begin{alignat}{2}
a\big( (\myvec{u},\mymatrix{p}),(\myvec{v},\mymatrix{q}) \big) + (\mymatrix{\lambda}, \mymatrix{q})_{0,\Omega} &= \ell(\myvec{v})
& \qquad &\forall \, (\myvec{v},\mymatrix{q})\in V \times Q, \label{eq:mixed_variationalF_01}\\
(\mymatrix{\mu}-\mymatrix{\lambda}, \mymatrix{p})_{0,\Omega} &\leq 0
& \qquad &\forall \, \mymatrix{\mu}\in\Lambda. \label{eq:mixed_variationalF_02}
\end{alignat}
\end{subequations} 

Let us summarize some fundamental properties of the weak formulations \eqref{eq:variq_second_kind} and \eqref{eq:mixed_variationalF}: First of all, it is well known, see e.g.~\cite{Han2013}, that for $\myvec{f}\in V^*$ and $\myvec{g}\in H^{-1/2}(\Gamma_N,\myspace{R}^d)$ there exists a unique solution $(\myvec{u},\mymatrix{p})\in V\times Q$ of the variational inequality \eqref{eq:variq_second_kind}. Moreover, it is shown in \cite{Bammer2023Apriori} that the formulations \eqref{eq:variq_second_kind} and \eqref{eq:mixed_variationalF} are equivalent in the sense that if $(\myvec{u},\mymatrix{p})\in V\times Q$ solves \eqref{eq:variq_second_kind}, then $(\myvec{u},\mymatrix{p},\mymatrix{\lambda})$ with
\begin{align}\label{eq:repres_continuousLambda}
 \mymatrix{\lambda} 
 = \operatorname{dev}\big( \mymatrix{\sigma}(\myvec{u},\mymatrix{p})-\myspace{H} \, \mymatrix{p} \big)
\end{align}
is a solution to \eqref{eq:mixed_variationalF} and, conversely, if $(\myvec{u},\mymatrix{p},\mymatrix{\lambda})\in V\times Q \times \Lambda$ solves \eqref{eq:mixed_variationalF}, then $(\myvec{u},\mymatrix{p})$ is a solution of \eqref{eq:variq_second_kind} and there holds the identity \eqref{eq:repres_continuousLambda}. As a consequence, there exists a unique solution to the mixed variational problem \eqref{eq:mixed_variationalF} and it holds
\begin{align}\label{eq:plambda}
 \mymatrix{p} : \mymatrix{\lambda} = \sigma_y \, \vert\mymatrix{p}\vert_F \quad
 \text{a.e.~in } \Omega,
\end{align}
cf.~\cite{Bammer2023Apriori}. Finally, by \cite{Bammer2023Apriori}, the solution $(\myvec{u},\mymatrix{p},\mymatrix{\lambda})\in V\times Q \times \Lambda$ of the the mixed variational problem \eqref{eq:mixed_variationalF} depends Lipschitz-continuously on the data  $\myvec{f}$, $\myvec{g}$ and $\sigma_y$. More precisely, there holds
\begin{align}\label{eq:Lipschitz_depend}
 \Vert (\myvec{u}_2-\myvec{u}_1, \mymatrix{p}_2-\mymatrix{p}_1) \Vert + \Vert \mymatrix{\lambda}_2-\mymatrix{\lambda}_1 \Vert_{0,\Omega}  
 \lesssim  \Vert \sigma_{y,2}-\sigma_{y,1} \Vert_{0,\Omega} + \Vert \myvec{f}_2-\myvec{f}_1\Vert_{V^*} + \Vert \myvec{g}_2-\myvec{g}_1\Vert_{H^{-1/2}(\Gamma_N,\myspace{R}^d)}, 
\end{align}
where $(\myvec{u}_i,\mymatrix{p}_i,\mymatrix{\lambda}_i)$, for $i=1,2$, is the solution to the data $(\myvec{f}_i,\myvec{g}_i,\sigma_{y,i})$.

\medskip

% ---------------------------------------------------------------------------------------
%   SECTION: Possible discretizations
% ----------------------------------------------------------------------------------------

\section{$hp$-Finite Element Discretizations} 
\label{sec:discretization}

For the discretization with $hp$-finite elements let $\mathcal{T}_{h}$ be a locally quasi-uniform finite element mesh of $\Omega$ consisting of convex and shape regular quadrilaterals or hexahedrons, respectively. Moreover, let $\widehat{T}:=[-1,1]^d$ be the reference element and let $\myvec{F}_T:\widehat{T}\rightarrow T$ denote the bi/tri-linear bijective mapping for $T\in\mathcal{T}_h$.
We set $h := (h_T)_{T\in\mathcal{T}_h}$ and $p := (p_T)_{T\in\mathcal{T}_h}$ where $h_T$ and $p_T$ denote the local element size and the local polynomial degree, respectively. We assume that the local polynomial degrees of neighboring elements are comparable and refer to \cite{Melenk2005} for details on quasi-uniformity and comparable polynomial degrees. Some of the following results exploit the exactness of the Gauss quadrature for polynomials. In these cases we additionally assume that
\begin{align}\label{eq:requirement_detF}
 \det \nabla \mathfrak{F}_T \in \myspace{P}_1(\widehat{T}) \qquad 
 \forall \, T\in\mathcal{T}_h \text{ with } p_T\geq 2.
\end{align}
Note that  $\det \nabla\mathfrak{F}_T$ has no change of sign in $\widehat{T}$. While \eqref{eq:requirement_detF} is not a restriction for lower-order methods or the two dimensional case it slightly limits the shape of mesh elements in the case that $d=3$. For the discretization of the displacement field and of the plastic strain we use the $hp$-finite element spaces
\begin{align*}
 V_{hp} &:= \Big\lbrace \myvec{v}_{hp}\in V \; ; \; \myvec{v}_{hp \, | \, T}\circ \myvec{F}_T\in \big(\myspace{P}_{p_T}(\widehat{T})\big)^d \text{ for all } T\in\mathcal{T}_h\Big\rbrace, \\
 Q_{hp} &:= \Big\lbrace \mymatrix{q}_{hp}\in Q \; ; \; \mymatrix{q}_{hp \, | \, T}\circ \myvec{F}_T\in \big(\myspace{P}_{p_T-1}(\widehat{T})\big)^{d\times d} \text{ for all } T\in\mathcal{T}_h\Big\rbrace.
\end{align*}
Furthermore, let $\hat{\myvec{x}}_{k,T}\in \widehat{T}$ for $1\leq k \leq n_T$ be the tensor product Gauss quadrature points on $\widehat{T}$ and $\hat{\omega}_{k,T}\in\myspace{R}$ the corresponding positive weights where $n_T := p_T^d$ for $T\in\mathcal{T}_h$. Thereby, we introduce the mesh dependent quadrature rule
\begin{align*}
 \mathcal{Q}_{hp}(\cdot) := \sum_{T\in\mathcal{T}_h} \mathcal{Q}_{hp, T}(\cdot),
\end{align*}
where the local quantities $\mathcal{Q}_{hp, T}(\cdot)$ are given by
\begin{align*}
 \mathcal{Q}_{hp, T}(f) :=
 \begin{cases}
  |T| \, f\big(\mathfrak{F}_T(\myvec{0})\big), & \text{if } p_T = 1, \\
  \sum_{k=1}^{n_T} \hat{\omega}_{k,T} \, |\det \nabla \mathfrak{F}_T(\hat{\mathfrak{x}}_{k,T})| \, f\big(\mathfrak{F}_T(\hat{\myvec{x}}_{k,T})\big), & \text{if } p_T \geq 2,
 \end{cases}
 \qquad T\in\mathcal{T}_h.
\end{align*}
Here, $|T|$ denotes the $d$-dimensional Lebesque-measure of $T$. On elements $T\in\mathcal{T}_h$ with $p_T \geq 2$ the quadrature rule $\mathcal{Q}_{hp}(\cdot)$ represents the standard Gauss quadrature on the reference element, whereas $\mathcal{Q}_{hp}(\cdot)$ is the midpoint rule on elements $T\in\mathcal{T}_h$ with $p_T=1$. Note that the Gauss quadrature on the reference element with only one point is not exact in the case that $\det \nabla\mathfrak{F}_T$ is a polynomial of degree $\geq 2$; even for constant $f$. 

\medskip

% ----------------------------------------------------------------------------------------
\subsection{Discretization of the variational inequality}

We use the quadrature rule $\mathcal{Q}_{hp}(\cdot)$ to approximate the plasticity functional $\psi(\cdot)$, which is appearing in \eqref{eq:variq_second_kind}, leading to the discrete plasticity functional given by
\begin{align} \label{eq:psi_ha_as_quadrature}
\psi_{hp}(\mymatrix{q}_{hp}) := \mathcal{Q}_{hp}\big( \psi(\mymatrix{q}_{hp}) \big) 
= \mathcal{Q}_{hp} \big( \sigma_y \, \vert \mymatrix{q}_{hp}\vert_F \big)
\end{align}
as proposed, e.g.~in \cite{gwinner2009p} for Tresca friction. Thereby, the discrete variational inequality is to find a pair $(\myvec{u}_{hp}, \mymatrix{p}_{hp})\in V_{hp}\times Q_{hp}$ such that
\begin{align}\label{eq:discrete_variq_second_kind}
 a\big( (\myvec{u}_{hp},\mymatrix{p}_{hp}), (\myvec{v}_{hp}-\myvec{u}_{hp}, \mymatrix{q}_{hp}-\mymatrix{p}_{hp}) \big) + \psi_{hp}(\mymatrix{q}_{hp}) - \psi_{hp}(\mymatrix{p}_{hp}) 
 \geq \ell(\myvec{v}_{hp} - \myvec{u}_{hp})
 \qquad \forall (\myvec{v}_{hp},\mymatrix{q}_{hp})\in V_{hp}\times Q_{hp}.
\end{align}
%for all $(\myvec{v}_{hp},\mymatrix{q}_{hp})\in V_{hp}\times Q_{hp}$.

\medskip

% ----------------------------------------------------------------------------------------
\subsection{Two discretizations of the mixed variational formulation}\label{subsec:two_discr_mixed}

Alternatively, we may discretize the mixed variational formulation \eqref{eq:mixed_variationalF}. In this section, we present two such discretizations that differ only in the choice of the set $\Lambda_{hp}$ of admissible discrete Lagrange multipliers. For this purpose, we introduce
\begin{align}
 \Lambda_{hp}^{(s)} &:= \Big\lbrace \mymatrix{\mu}_{hp}\in Q_{hp} \; ; \; \big\vert \mymatrix{\mu}_{hp} \big( \myvec{F}_T (\hat{\myvec{x}}_{k,T})\big) \big\vert_F \leq \sigma_y \text{ for all } 1\leq k \leq n_T \text{ and } T \in \mathcal{T}_h \Big\rbrace, \\
  \Lambda_{hp}^{(w)} &:= \Big\lbrace \mymatrix{\mu}_{hp}\in Q_{hp} \; ; \; (\mymatrix{\mu}_{hp},\mymatrix{q}_{hp})_{0,\Omega} \leq \psi_{hp}(\mymatrix{q}_{hp}) \text{ for all } \mymatrix{q}_{hp}\in Q_{hp}  \Big\rbrace.
\end{align}
Thereby, the two discrete mixed formulations are: Find a triple $(\myvec{u}_{hp},\mymatrix{p}_{hp},\mymatrix{\lambda}_{hp})\in V_{hp}\times Q_{hp}\times\Lambda_{hp}^{(i)}$, for $i \in \{s,w\}$, such that
\begin{subequations}\label{eq:discrete_mixed_variationalF}
\begin{alignat}{2}
 a\big( (\myvec{u}_{hp},\mymatrix{p}_{hp}),(\myvec{v}_{hp},\mymatrix{q}_{hp}) \big) + (\mymatrix{\lambda}_{hp}, \mymatrix{q}_{hp})_{0,\Omega} &= \ell(\myvec{v}_{hp}) & \qquad & \forall \, (\myvec{v}_{hp},\mymatrix{q}_{hp}) \in V_{hp}\times Q_{hp},\label{eq:discrete_mixed_variationalF_01}\\
 (\mymatrix{\mu}_{hp}-\mymatrix{\lambda}_{hp}, \mymatrix{p}_{hp})_{0,\Omega}
 &\leq 0 & \quad & \forall \, \mymatrix{\mu}_{hp} \in \Lambda_{hp}^{(i)}.\label{eq:discrete_mixed_variationalF_02}
\end{alignat}
\end{subequations}

%We emphasize that the discrete variational inequality \eqref{eq:discrete_variq_second_kind} and the discrete mixed formulations \eqref{eq:discrete_mixed_variationalF} for $i\in\lbrace s,w\rbrace$ are closely related, which is pointed out in the subsequent section.

\medskip

% ----------------------------------------------------------------------------------------
\subsection{Relation between the three discretizations}

In this section we prove the equivalence of the discrete variational inequality  \eqref{eq:discrete_variq_second_kind} and the discrete mixed formulation \eqref{eq:discrete_mixed_variationalF} with the specific choice $\Lambda_{hp}^{(i)}=\Lambda_{hp}^{(w)}$ (even if \eqref{eq:requirement_detF} does not hold). We note that this result can be generalized to other element shapes such as triangles, tetrahedrons or pyramids. It is shown in \cite{Bammer2023Apriori} that under the assumption \eqref{eq:requirement_detF} there holds $\Lambda_{hp}^{(s)}  = \Lambda_{hp}^{(w)} $. As a consequence, the two discrete mixed variational formulations \eqref{eq:discrete_mixed_variationalF} coincide in this case.
\vspace{0.2cm}

Let $\{\widehat{\phi}_{k,T}\}_{k=1,\ldots,n_T}$ be the Lagrange basis functions on $\widehat{T}$ defined via the Gauss points $\hat{\myvec{x}}_{l,T}$, i.e.
\begin{align*}
 \widehat{\phi}_{k,T}\in\myspace{P}_{p_T-1}(\widehat{T}), \quad
 \widehat{\phi}_{k,T}(\hat{\myvec{x}}_{l,T})=\delta_{kl} \qquad
 \forall \, 1\leq k,l\leq n_T \quad \forall \, T\in\mathcal{T}_h,
\end{align*}
where $\delta_{kl}$ is the usual Kronecker delta symbol. Moreover, let $\phi_1,\ldots,\phi_N$ be piecewisely defined as
\begin{align*}
 \phi_{\zeta(k,T')\, | \, T} 
 := \begin{cases}
                \widehat{\phi}_{k,T'}\circ\myvec{F}_{T'}^{-1},& \text{if } T=T',\\
                0,& \text{if } T\neq T', \\
    \end{cases} 
 \qquad  T,T'\in\mathcal{T}_h, \quad 1\leq k\leq n_T,
\end{align*}
where $\zeta:\big\lbrace (k,T) \; ; \;  T\in\mathcal{T}_h, \, 1\leq k\leq n_T \big\rbrace\rightarrow \{1,\ldots,N\}$ with $N :=\sum_{T\in\mathcal{T}_h} n_T$ is a one to one numbering. Obviously, $\lbrace\phi_1,\ldots,\phi_N\rbrace$ forms a basis of the $hp$-finite element space
\begin{align*}
 W_{hp}:=\left\{q_{hp}\in L^2(\Omega) \; ; \; q_{hp \, | \, T}\circ \myvec{F}_T\in \myspace{P}_{p_T-1}(\widehat{T})  \text{ for all } T\in\mathcal{T}_h\right\}.
\end{align*}
Furthermore, let $\varphi_1,\ldots,\varphi_N$ be the biorthogonal basis functions to $\phi_1,\ldots,\phi_N$ that are uniquely determined by the conditions $\varphi_{\zeta(k,T)\, | \, T} \circ\myvec{F}_T\in\myspace{P}_{p_T-1}(\widehat{T})$ for $T\in\mathcal{T}_h$, $1\leq k\leq n_T$ and 
\begin{align*}
 (\phi_i,\varphi_j)_{0,\Omega} = \delta_{ij} \, (\phi_i, 1)_{0,\Omega} \qquad 
 \forall\, 1\leq i,j \leq N.
\end{align*}
As the functions $\varphi_1,\ldots,\varphi_N$ are linearly independent they form a basis of $W_{hp}$ as well and there holds
\begin{align*}
 \operatorname{supp}(\varphi_i) = \operatorname{supp}(\phi_i) \qquad
 \forall \, 1\leq i \leq N.
\end{align*}
Under the assumption \eqref{eq:requirement_detF} we have $\varphi_{\zeta(k,T)}=\phi_{\zeta(k,T)}$ for $1\leq k\leq n_T$ and $T\in\mathcal{T}_h$. We refer to \cite{banz2015biorthogonal} for the computation of connectivity matrices to construct the global biorthogonal basis functions in the case that \eqref{eq:requirement_detF} does not hold true.
Note that for all $1\leq i\leq N$ the quantities $D_i:=(\phi_i, 1)_{0,\Omega}$ are positive as $\phi_i$ is a Gauss-Legendre-Lagrange basis function and let
\begin{align}\label{eq:defi_sigmai}
 \sigma_i := D_i^{-1} (\sigma_y,\phi_i)_{0,\Omega} \qquad 
 \forall \, 1\leq i\leq N.
\end{align}
Since $\sigma_y$ is assumed to be a constant we have $\sigma_i = \sigma_y > 0$ for $1\leq i\leq N$. In view of the Lipschitz dependency of the solution on $\sigma_y$, see \eqref{eq:Lipschitz_depend}, the assumption that $\sigma_y$ is a constant can be weakened as long as $\sigma_i>0$ for any $1\leq i\leq N$, which is, for instance, the case if $\sigma_{y\, | \, T} \circ\myvec{F}_T\in\myspace{P}_{\max(p_T+1-d,0)}(\widehat{T})$. \\

The previous discussions yield the following equivalent representations
\begin{align}\label{eq:representation_Qhp}
 Q_{hp} 
 = \left\{ \sum_{i=1}^N \mymatrix{q}_i \, \phi_i \; ; \; \mymatrix{q}_i\in\myspace{S}_{d,0} \right\} 
 = \left\{ \sum_{i=1}^N \mymatrix{\mu}_i \, \varphi_i \; ; \; \mymatrix{\mu}_i\in\myspace{S}_{d,0} \right\},
\end{align}
which, in fact, allows the decoupling of the constraints in $\Lambda_{hp}^{(w)}$ and \eqref{eq:discrete_mixed_variationalF_02}: It is shown in \cite{Bammer2022Icosahom} that there holds
\begin{align}\label{eq:characterization_LagrangeMult}
 \Lambda_{hp}^{(w)} 
 = \left\{ \sum_{i=1}^N \mymatrix{\mu}_i \, \varphi_i \; ; \;  \mymatrix{\mu}_i \in \myspace{S}_{d,0} \text{ and } \vert \mymatrix{\mu}_i \vert_F \leq \sigma_i\right\}.
\end{align}
Furthermore, representing $\mymatrix{\lambda}_{hp}\in \Lambda_{hp}^{(w)}$ and $\mymatrix{p}_{hp}\in Q_{hp}$ as
\begin{align*}
 \mymatrix{\lambda}_{hp} = \sum_{i=1}^N \mymatrix{\lambda}_i \, \varphi_i, \quad
 \mymatrix{p}_{hp} = \sum_{i=1}^N \mymatrix{p}_i \, \phi_i,
\end{align*}
respectively, we observe that $\mymatrix{\lambda}_{hp}$ satisfies the inequality \eqref{eq:discrete_mixed_variationalF_02} if and only if 
\begin{align}\label{eq:char_IQconstraint}
 \mymatrix{\lambda}_i : \mymatrix{p}_i 
 = \sigma_i \, \vert \mymatrix{p}_i \vert_F \qquad
 \forall \, 1\leq i \leq N.
\end{align}

\begin{lemma}\label{lem:representation_Qhp_qhp}
For $\mymatrix{q}_{hp} = \sum_{i=1}^N \mymatrix{q}_i \, \phi_i\in Q_{hp}$ there holds
\begin{align}\label{eq:representation_Qhp}
 \psi_{hp}(\mymatrix{q}_{hp}) 
 = \sum_{i=1}^N | \mymatrix{q}_i |_F \, (\sigma_y, \phi_i)_{0,\Omega}.
\end{align}
\end{lemma}

\begin{proof}
We first note that by using the numbering $\zeta$ we have for $T\in\mathcal{T}_h$
\begin{align*}
 \mymatrix{q}_{hp \, | \, T} 
 = \left. \left( \sum_{i=1}^N \mymatrix{q}_i \, \phi_i \right) \right|_{\, T}
 = \left. \left( \sum_{T'\in\mathcal{T}_h} \sum_{l=1}^{n_{T'}} \mymatrix{q}_{l,T'} \, \phi_{l,T'} \right) \right|_{\, T}
 = \sum_{l=1}^{n_T} \mymatrix{q}_{l,T} \, \phi_{l,T}
\end{align*}
as $\operatorname{supp}(\phi_{l,T'}) = T'$, where $i=\zeta(l,T')$. If $p_T = 1$ then $ \mymatrix{q}_{hp \, | \, T}$ is constant and, therefore,
\begin{align*} %\label{eq:proofRepresqhp_01}
 \mathcal{Q}_{hp, T}\big( \sigma_y \, \big| \mymatrix{q}_{hp \, | \, T} \big|_F \big) 
 = |T| \, \sigma_y \, \big| \mymatrix{q}_{hp \, | \, T}  \big( \mathfrak{F}_T ( \myvec{0}) \big) \big|_F
 = |\mymatrix{q}_{1,T}|_F \, \sigma_y \, |T|
 = |\mymatrix{q}_{1,T}|_F \, (\sigma_y, \phi_{1,T})_{0,T}
\end{align*}
since $\phi_{1,T}\equiv 1$. For $p_T\geq 2$ the exactness of the Gauss quadrature leads to
\begin{align*}
 \mathcal{Q}_{hp, T}\big( \sigma_y \, \big| \mymatrix{q}_{hp \, | \, T} \big|_F \big)
 &= \sum_{k=1}^{n_T} \hat{\omega}_{k,T} \, \big| \det \nabla \mathfrak{F}_T(\hat{\myvec{x}}_{k,T}) \big| \, \sigma_y \, \big| \mymatrix{q}_{hp \, | \, T} \big(\mathfrak{F}_T ( \hat{\myvec{x}}_{k,T}) \big) \big|_F \\
 &= \sum_{k=1}^{n_T} \hat{\omega}_{k,T} \, \big| \det \nabla \mathfrak{F}_T(\hat{\myvec{x}}_{k,T}) \big| \, \sigma_y \, \left| \sum_{l=1}^{n_T} \mymatrix{q}_{l,T} \, \phi_{l,T} \big(\mathfrak{F}_T ( \hat{\myvec{x}}_{k,T}) \big) \right|_F \\
 &= \sum_{k=1}^{n_T} \hat{\omega}_{k,T} \, \big| \det \nabla \mathfrak{F}_T(\hat{\myvec{x}}_{k,T}) \big| \, \sigma_y \, \big| \mymatrix{q}_{k,T} \big|_F \, \phi_{k,T} \big(\mathfrak{F}_T ( \hat{\myvec{x}}_{k,T}) \big) \\
 &= \sum_{k=1}^{n_T} | \mymatrix{q}_{k,T} |_F \sum_{l=1}^{n_T} \hat{\omega}_{l,T} \, \big| \det \nabla \mathfrak{F}_T(\hat{\myvec{x}}_{l,T}) \big| \, \sigma_y \, \phi_{k,T} \big(\mathfrak{F}_T ( \hat{\myvec{x}}_{l,T}) \big) \\
 &= \sum_{k=1}^{n_T} | \mymatrix{q}_{k,T} |_F \, (\sigma_y, \phi_{k,T})_{0,T}
\end{align*}
as $\phi_{l,T}\big(\mathfrak{F}_T(\hat{\myvec{x}}_{k,T})\big) = \delta_{kl}$ for $1\leq k,l\leq n_T$. Thus,
\begin{align*} %\label{eq:proofRepresqhp_02}
 \mathcal{Q}_{hp}\big( \sigma_y \, |\mymatrix{q}_{hp}|_F \big)
 = \sum_{T\in\mathcal{T}_h} \mathcal{Q}_{hp, T}\big( \sigma_y \, \big| \mymatrix{q}_{hp \, | \, T} \big|_F \big)
 = \sum_{T\in\mathcal{T}_h} \sum_{k=1}^{n_T} | \mymatrix{q}_{k,T} |_F \, (\sigma_y, \phi_{k,T})_{0,T}
 = \sum_{i=1}^N | \mymatrix{q}_i |_F \, (\sigma_y, \phi_i)_{0,\Omega},
\end{align*}
which is \eqref{eq:representation_Qhp} exploiting that $\psi_{hp}(\mymatrix{q}_{hp}) = \mathcal{Q}_{hp} \big( \sigma_y \, \vert \mymatrix{q}_{hp}\vert_F \big)$.
\end{proof}

With $\mathcal{P}_{hp}: Q \rightarrow Q_{hp}$ denoting the standard $L^2$-projection operator the equivalence of \eqref{eq:discrete_variq_second_kind} and  \eqref{eq:discrete_mixed_variationalF} with $\Lambda_{hp}^{(i)}=\Lambda_{hp}^{(w)}$ is stated in the following theorem.

\begin{theorem}\label{cor:equivalence_discreteCase}
If the pair $(\myvec{u}_{hp},\mymatrix{p}_{hp})\in V_{hp}\times Q_{hp}$ solves \eqref{eq:discrete_variq_second_kind}, then the triple $(\myvec{u}_{hp},\mymatrix{p}_{hp},\mymatrix{\lambda}_{hp})$ with 
\begin{align}\label{eq:repres_lagrangeM}
 \mymatrix{\lambda}_{hp} = \mathcal{P}_{hp}\left(\operatorname{dev}(\mymatrix{\sigma}(\myvec{u}_{hp},\mymatrix{p}_{hp}) - \myspace{H} \mymatrix{p}_{hp}) \right)
\end{align}
is a solution of \eqref{eq:discrete_mixed_variationalF} with $\Lambda_{hp}^{(i)}=\Lambda_{hp}^{(w)}$. Conversely, if $(\myvec{u}_{hp},\mymatrix{p}_{hp},\mymatrix{\lambda}_{hp})\in V_{hp}\times Q_{hp}\times \Lambda_{hp}^{(w)}$ solves \eqref{eq:discrete_mixed_variationalF} with $\Lambda_{hp}^{(i)}=\Lambda_{hp}^{(w)}$, then $(\myvec{u}_{hp},\mymatrix{p}_{hp})$ solves \eqref{eq:discrete_variq_second_kind} and there holds \eqref{eq:repres_lagrangeM}.
\end{theorem}

\begin{proof}
Let $(\myvec{u}_{hp},\mymatrix{p}_{hp},\mymatrix{\lambda}_{hp})\in  V_{hp}\times Q_{hp}\times \Lambda_{hp}^{(w)}$ be a solution of \eqref{eq:discrete_mixed_variationalF}. Using \eqref{eq:characterization_LagrangeMult} and \eqref{eq:representation_Qhp} we represent $\mymatrix{\lambda}_{hp}\in\Lambda_{hp}^{(w)}$ and $\mymatrix{p}_{hp}\in Q_{hp}$ by
\begin{align*}
 \mymatrix{\lambda}_{hp}=\sum_{i=1}^N \mymatrix{\lambda}_i \, \varphi_i, \quad
 \mymatrix{p}_{hp} = \sum_{i=1}^N \mymatrix{p}_i \, \phi_i.
\end{align*}
Thus, the identity \eqref{eq:char_IQconstraint}, the biorthogonality of the basis functions, the definition of $\sigma_i$, cf.~\eqref{eq:defi_sigmai}, and Lemma~\ref{lem:representation_Qhp_qhp} yield
\begin{align*}
 \left(\mymatrix{\lambda}_{hp},\mymatrix{p}_{hp} \right)_{0,\Omega} 
 = \sum_{i,j=1}^N \mymatrix{\lambda}_i : \mymatrix{p}_j \, (\varphi_i, \phi_j)_{0,\Omega}
 = \sum_{i=1}^N \mymatrix{\lambda}_i : \mymatrix{p}_i \, D_i 
 = \sum_{i=1}^N \sigma_i \, \vert \mymatrix{p}_i \vert_F \, D_i 
 = \sum_{i=1}^N \vert \mymatrix{p}_i \vert_F  \, (\sigma_y,\phi_i)_{0,T}
 = \psi_{hp}(\mymatrix{p}_{hp}).
\end{align*}
As $\mymatrix{\lambda}_{hp} \in \Lambda_{hp}^{ (w)}$ we have $(\mymatrix{\lambda}_{hp},\mymatrix{q}_{hp})_{0,\Omega} \leq \psi_{hp}(\mymatrix{q}_{hp})$ for any $\mymatrix{q}_{hp}\in Q_{hp}$. Hence,
\begin{align*}
 (\mymatrix{\lambda}_{hp},\mymatrix{q}_{hp} - \mymatrix{p}_{hp})_{0,\Omega}
 = (\mymatrix{\lambda}_{hp},\mymatrix{q}_{hp})_{0,\Omega} - \psi_{hp}(\mymatrix{p}_{hp})
 \leq \psi_{hp}(\mymatrix{q}_{hp}) - \psi_{hp}(\mymatrix{p}_{hp}).
\end{align*}
Thereby, choosing $(\myvec{v}_{hp}-\myvec{u}_{hp}, \mymatrix{q}_{hp}-\mymatrix{p}_{hp})$ as test function in \eqref{eq:discrete_mixed_variationalF_01} with arbitrary $(\myvec{v}_{hp}, \mymatrix{q}_{hp}) \in V_{hp} \times Q_{hp}$ yields \eqref{eq:discrete_variq_second_kind}. To show \eqref{eq:repres_lagrangeM}, we choose $\myvec{v}_{hp} = \myvec{0}$ in \eqref{eq:discrete_mixed_variationalF_01} to obtain
\begin{align}\label{eq:construction_lambda_hr}
 ( \mymatrix{\lambda}_{hp},\mymatrix{q}_{hp} )_{0,\Omega} 
 = -a\big( (\myvec{u}_{hp},\mymatrix{p}_{hp}),(\myvec{0},\mymatrix{q}_{hp}) \big) \qquad
 \forall \, \mymatrix{q}_{hp}\in Q_{hp},
\end{align}
from which we deduce
\begin{align}\label{eq:proof_lambda_h_proj_of_dev}
 ( \mymatrix{\lambda}_{hp},\mymatrix{q}_{hp} )_{0,\Omega}
 = \big( \operatorname{dev} \big( \mymatrix{\sigma}(\myvec{u}_{hp},\mymatrix{p}_{hp}) - \myspace{H} \, \mymatrix{p}_{hp} \big), \mymatrix{q}_{hp} \big)_{0,\Omega} \qquad
 \forall \, \mymatrix{q}_{hp}\in Q_{hp}
\end{align}
by using the definition \eqref{eq:defi_bilinearF_a} of the bilinearform $a(\cdot,\cdot)$ and the identity \eqref{eq:identity_deviatoricPart}. Indeed, this gives
\begin{align*}
 \mymatrix{\lambda}_{hp} = \mathcal{P}_{hp}\left(\operatorname{dev}(\mymatrix{\sigma}(\myvec{u}_{hp},\mymatrix{p}_{hp}) - \myspace{H} \mymatrix{p}_{hp}) \right).
\end{align*}

Let $(\myvec{u}_{hp},\mymatrix{p}_{hp})$ be a solution of \eqref{eq:discrete_variq_second_kind} and define $\mymatrix{\lambda}_{hp} := \mathcal{P}_{hp}\left(\operatorname{dev}(\mymatrix{\sigma}(\myvec{u}_{hp},\mymatrix{p}_{hp}) - \myspace{H} \mymatrix{p}_{hp}) \right)$. By using \eqref{eq:identity_deviatoricPart} and the definition \eqref{eq:defi_bilinearF_a} of $a(\cdot,\cdot)$ we obtain \eqref{eq:construction_lambda_hr}. Choosing \eqref{eq:discrete_variq_second_kind} with the test functions $(\myvec{u}_{hp} \pm \myvec{v}_{hp}, \mymatrix{p}_{hp})$ for arbitrary $\myvec{v}_{hp}\in V_{hp}$ leads to the equation
\begin{align*}
 a\big( (\myvec{u}_{hp},\myvec{v}_{hp}), (\myvec{v}_{hp},\mymatrix{0}) \big) = \ell(\myvec{v}_{hp}) \qquad
 \forall \, \myvec{v}_{hp}\in V_{hp}.
\end{align*}
Furthermore, by using \eqref{eq:construction_lambda_hr} we find that 
\begin{align*}
 a\big( (\myvec{u}_{hp},\myvec{v}_{hp}), (\myvec{v}_{hp},\mymatrix{q}_{hp}) \big) + ( \mymatrix{\lambda}_{hp},\mymatrix{q}_{hp} )_{0,\Omega} 
 = a\big( (\myvec{u}_{hp},\myvec{v}_{hp}), (\myvec{v}_{hp},\mymatrix{0}) \big) \qquad
 \forall \, \mymatrix{q}_{hp}\in Q_{hp},
\end{align*}
which gives \eqref{eq:discrete_mixed_variationalF_01}. Choosing $\mymatrix{q}_{hp}-\mymatrix{p}_{hp}\in Q_{hp}$ and $(\myvec{u}_{hp},\mymatrix{q}_{hp})\in V_{hp}\times Q_{hp}$ as test function in \eqref{eq:construction_lambda_hr} and \eqref{eq:discrete_variq_second_kind}, respectively, implies that
\begin{align}\label{eq:proofThm8_01}
 (\mymatrix{\lambda}_{hp},\mymatrix{q}_{hp}-\mymatrix{p}_{hp})_{0,\Omega}
 = -a\big( (\myvec{u}_{hp},\mymatrix{p}_{hp}),(\myvec{0},\mymatrix{q}_{hp}-\mymatrix{p}_{hp}) \big)
 \leq \psi_{hp}(\mymatrix{q}_{hp}) - \psi_{hp}(\mymatrix{p}_{hp}) \qquad 
 \forall \, \mymatrix{q}_{hp} \in Q_{hp}.
\end{align}
In particular, inserting $\mymatrix{q}_{hp} = \mymatrix{0}$ and $\mymatrix{q}_{hp} = 2 \, \mymatrix{p}_{hp}$ gives
\begin{align}\label{eq:proofThm8_02}
 ( \mymatrix{\lambda}_{hp},\mymatrix{p}_{hp} )_{0,\Omega} = \psi_{hp}(\mymatrix{p}_{hp}),
\end{align}
which turns \eqref{eq:proofThm8_01} into 
\begin{align}\label{eq:proofThm8_03}
 (\mymatrix{\lambda}_{hp},\mymatrix{q}_{hp})_{0,\Omega} \leq \psi_{hp}(\mymatrix{q}_{hp}) \qquad 
 \forall \, \mymatrix{q}_{hp} \in Q_{hp}.
\end{align}
Thus, it holds $\mymatrix{\lambda}_{hp} \in \Lambda_{hp}^{(w)}$. Finally, from \eqref{eq:proofThm8_02} and \eqref{eq:proofThm8_03} we deduce
\begin{align*}
 \left(\mymatrix{\mu}_{hp}-\mymatrix{\lambda}_{hp},\mymatrix{p}_{hp} \right)_{0,\Omega}
 \leq \psi_{hp}(\mymatrix{p}_{hp})-\left(\mymatrix{\lambda}_{hp},\mymatrix{p}_{hp} \right)_{0,\Omega}
 = 0 \qquad
 \forall \, \mymatrix{\mu}_{hp}\in\Lambda_{hp}^{(w)},
\end{align*}
which is \eqref{eq:discrete_mixed_variationalF_02}.
\end{proof}

% ----------------------------------------------------------------------------------------
\subsection{Existence of discrete solutions}

The existence of a unique solution $(\myvec{u}_{hp},\mymatrix{p}_{hp},\mymatrix{\lambda}_{hp})$ of \eqref{eq:discrete_mixed_variationalF} with the specific choice $\Lambda_{hp}^{(i)}=\Lambda_{hp}^{(s)}$ is already stated in \cite{Bammer2023Apriori}. The existence of a solution $(\myvec{u}_{hp},\mymatrix{p}_{hp},\mymatrix{\lambda}_{hp})$ of \eqref{eq:discrete_mixed_variationalF} with $\Lambda_{hp}^{(i)}=\Lambda_{hp}^{(w)}$ can be deduced from the existence of a solution $(\myvec{u}_{hp},\mymatrix{p}_{hp})\in V_{hp}\times Q_{hp}$ of \eqref{eq:discrete_variq_second_kind} due to their equivalence as stated in Theorem~\ref{cor:equivalence_discreteCase}. Note that the uniqueness of the Lagrange multiplier $\mymatrix{\lambda}_{hp}\in\Lambda_{hp}^{(w)}$ directly follows from the discrete inf-sup condition 
\begin{align}\label{eq:discr_infSup}
 \sup_{\substack{(\myvec{v}_{hp},\mymatrix{q}_{hp})\in V_{hp}\times Q_{hp} \\ \Vert (\myvec{v}_{hp},\mymatrix{q}_{hp})\Vert \neq 0}} \frac{(\mymatrix{\mu}_{hp},\mymatrix{q}_{hp})_{0,\Omega}}{\Vert(\myvec{v}_{hp},\mymatrix{q}_{hp})\Vert}
 = \Vert \mymatrix{\mu}_{hp}\Vert_{0,\Omega} \qquad 
 \forall\, \mymatrix{\mu}_{hp}\in \Lambda_{hp}^{(i)}
\end{align}
%which holds true for all $\mymatrix{\mu}_{hp}\in\Lambda_{hp}^{(i)}$ as $\Lambda_{hp}^{(i)}\subseteq Q_{hp}$ 
for $i\in\lbrace s,w\rbrace$. A proof of \eqref{eq:discr_infSup} can be found in \cite{Bammer2023Apriori}. To show the unique existence of a solution $(\myvec{u}_{hp},\mymatrix{p}_{hp})\in V_{hp}\times Q_{hp}$ of \eqref{eq:discrete_variq_second_kind} let us first state the subdifferentiability of the discrete plasticity functional $\psi_{hp}(\cdot)$.

\begin{lemma}\label{lem:subdiff_psihp}
The approximative plasticity functional $\psi_{hp}(\cdot)$ is subdifferentiable on $Q_{hp}$, i.e.~for any $\mymatrix{\mu}_{hp}\in Q_{hp}$ there exists an element $\partial\psi_{hp}(\mymatrix{\mu}_{hp})$ in the dual space $Q_{hp}^{\ast}$ of $Q_{hp}$ such that
\begin{align*}
 \psi_{hp}(\mymatrix{q}_{hp}) \geq \psi_{hp}(\mymatrix{\mu}_{hp}) + \big\langle \partial\psi_{hp}(\mymatrix{\mu}_{hp}), \mymatrix{q}_{hp} - \mymatrix{\mu}_{hp} \big\rangle \qquad
 \forall \, \mymatrix{q}_{hp}\in Q_{hp}.
\end{align*}
\end{lemma}

\begin{proof}
Representing $\mymatrix{\nu}_{hp},\mymatrix{q}_{hp}\in Q_{hp}$ as
\begin{align*}
 \mymatrix{\nu}_{hp} = \sum_{i=1}^N \mymatrix{\nu}_i \, \varphi_i, \quad
 \mymatrix{q}_{hp} = \sum_{i=1}^N \mymatrix{q}_i \, \phi_i,
\end{align*}
cf.~\eqref{eq:representation_Qhp}, we conclude from the biorthgonality of the basis functions that
\begin{align}\label{eq:formula_L2product}
 (\mymatrix{\nu}_{hp},\mymatrix{q}_{hp})_{0,\Omega} 
 = \sum_{i,j=1}^N \mymatrix{\nu}_i : \mymatrix{q}_i \, (\varphi_i, \phi_j)_{0,\Omega}
 = \sum_{i,j=1}^N \mymatrix{\nu}_i : \mymatrix{q}_i \, D_i.
\end{align}
Let $\mymatrix{\mu}_{hp}\in Q_{hp}$ with $\mymatrix{\mu}_{hp} = \sum_{i=1}^N \mymatrix{\mu}_i \, \phi_i$. For $i\in I := \lbrace 1\leq i\leq N \; ; \; \mymatrix{\mu}_i \neq \mymatrix{0} \rbrace$, the Cauchy-Schwarz inequality yields $|\mymatrix{\mu}_i|_F^{-1} \, \mymatrix{\mu}_i \leq |\mymatrix{q}_i|_F$ from which we deduce
\begin{align}\label{eq:proof_subDiff_psihp}
 |\mymatrix{q}_i|_F - |\mymatrix{\mu}_i|_F
 \geq \frac{1}{\; |\mymatrix{\mu}_i|_F} \, \mymatrix{\mu}_i : (\mymatrix{q}_i - \mymatrix{\mu}_i). 
\end{align}
Since $(\sigma_y, \phi_i)_{0,\Omega} = \sigma_i \, D_i$ we obtain from Lemma~\ref{lem:representation_Qhp_qhp} and \eqref{eq:proof_subDiff_psihp} that
\begin{align*}%\label{eq:proof_subDiff_psihp_01}
 \psi_{hp}(\mymatrix{q}_{hp}) - \psi_{hp}(\mymatrix{\mu}_{hp}) 
 = \sum_{i=1}^N \sigma_i \, D_i \, \big( |\mymatrix{q}_i|_F - |\mymatrix{\mu}_i|_F \big)
 \geq \sum_{i\in I} \frac{\sigma_i}{\; |\mymatrix{\mu}_i|_F}  \, \mymatrix{\mu}_i : (\mymatrix{q}_i - \mymatrix{\mu}_i) \, D_i
 = ( \widetilde{\mymatrix{q}}, \mymatrix{q}_{hp} - \mymatrix{\mu}_{hp} )_{0,\Omega},
\end{align*}
where the last identity follows from \eqref{eq:formula_L2product} and $\widetilde{\mymatrix{q}}\in Q_{hp}$ defined as
\begin{align*}
 \widetilde{\mymatrix{q}} := \sum_{i\in I} \frac{\sigma_i}{\; |\mymatrix{\mu}_i|_F}  \, \mymatrix{\mu}_i \, \varphi_i.
\end{align*}
Thus, the subdifferential $\partial\psi_{hp}(\mymatrix{\mu}_{hp})$ is given by the Riesz representator of $\widetilde{\mymatrix{q}}$.
\end{proof}

\begin{theorem} \label{thm:discreteExistence}
The discrete variational inequality \eqref{eq:discrete_variq_second_kind} has a unique solution $(\myvec{u}_{hp},\mymatrix{p}_{hp})\in V_{hp}\times Q_{hp}$.
\end{theorem}

\begin{proof}
We conclude from the symmetry and the ellipticity of the bilinear form $a(\cdot,\cdot)$ as well as the convexity of the discrete plasticity functional $\psi_{hp}(\cdot)$ that the discrete variational inequality is equivalent to a certain minimization problem, i.e.~a pair $(\myvec{u}_{hp}, \mymatrix{p}_{hp})\in V_{hp}\times Q_{hp}$ is a solution to \eqref{eq:discrete_variq_second_kind} if and only if it is a minimizer of $\mathcal{E} : V_{hp}\times Q_{hp}\to \mathbb{R}$, given by
\begin{align*}
 \mathcal{E}(\myvec{v}_{hp},\mymatrix{q}_{hp}) :=
 \frac{1}{2} \, a\big( (\myvec{v}_{hp}, \mymatrix{q}_{hp}), (\myvec{v}_{hp}, \mymatrix{q}_{hp}) \big) + \psi_{hp}(\mymatrix{q}_{hp}) - \ell(\myvec{v}_{hp}),
\end{align*}
see e.g.~\cite{Han2013}. From Lemma~\ref{lem:subdiff_psihp} we deduce for $(\myvec{w}_{hp},\mymatrix{\mu}_{hp})\in V_{hp}\times Q_{hp}$ 
\begin{align*}
 \mathcal{E}(\myvec{v}_{hp}, \mymatrix{q}_{hp}) &- \mathcal{E}(\myvec{w}_{hp}, \mymatrix{\mu}_{hp}) \\
 & \quad \geq \big\langle A(\myvec{w}_{hp}, \mymatrix{\mu}_{hp}), (\myvec{v}_{hp} - \myvec{w}_{hp}, \mymatrix{q}_{hp} - \mymatrix{\mu}_{hp}) \big\rangle + \big\langle \partial\psi_{hp}(\mymatrix{\mu}_{hp}), \mymatrix{q}_{hp} - \mymatrix{\mu}_{hp} \big\rangle - \langle \ell, \myvec{v}_{hp} - \myvec{w}_{hp}\rangle
 %\qquad \forall (\myvec{v}_{hp},\mymatrix{q}_{hp})\in V_{hp}\times Q_{hp}
\end{align*}
for all $(\myvec{v}_{hp},\mymatrix{q}_{hp})\in V_{hp}\times Q_{hp}$, where the operator $A(\myvec{w}_{hp}, \mymatrix{\mu}_{hp})\in (V_{hp}\times Q_{hp})^*$ is uniquely determined by
\begin{align*}
 a\big( (\myvec{w}_{hp}, \mymatrix{\mu}_{hp}), (\myvec{v}_{hp}, \mymatrix{q}_{hp}) \big)
 = \big\langle A(\myvec{w}_{hp}, \mymatrix{\mu}_{hp}), (\myvec{v}_{hp}, \mymatrix{q}_{hp}) \big\rangle\qquad  
 \forall \, (\myvec{v}_{hp},\mymatrix{q}_{hp})\in V_{hp}\times Q_{hp}.
\end{align*}
Hence, $\mathcal{E}(\cdot)$ is coercive, convex and subdifferentiable, which implies the existence of a minimizer of $\mathcal{E}(\cdot)$, see \cite[Ch.~II, Prop.~1.2]{Ekeland_1976}. Therefore, the discrete variational inequality \eqref{eq:discrete_variq_second_kind} also has a solution. The uniqueness of a solution $(\myvec{u}_{hp},\mymatrix{p}_{hp})\in V_{hp}\times Q_{hp}$ of \eqref{eq:discrete_variq_second_kind} follows directly from the ellipticity of $a(\cdot,\cdot)$. To see this, let $(\myvec{u}_{hp},\mymatrix{p}_{hp})$, $(\widetilde{\myvec{u}}_{hp},\widetilde{\mymatrix{p}}_{hp})\in V_{hp}\times Q_{hp}$ be two solutions. Then, adding the two associated inequalities resulting from \eqref{eq:discrete_variq_second_kind},
\begin{align*}
 a\big( (\myvec{u}_{hp},\mymatrix{p}_{hp}), (\widetilde{\myvec{u}}_{hp} - \myvec{u}_{hp},\widetilde{\mymatrix{p}}_{hp} - \mymatrix{p}_{hp}) \big) + \psi_{hp}(\widetilde{\mymatrix{p}}_{hp}) - \psi_{hp}(\mymatrix{p}_{hp}) 
 \geq \ell(\widetilde{\myvec{u}}_{hp} - \myvec{u}_{hp})
\end{align*}
and
\begin{align*}
  a\big( (\widetilde{\myvec{u}}_{hp}, \widetilde{\mymatrix{p}}_{hp}), (\myvec{u}_{hp} - \widetilde{\myvec{u}}_{hp}, \mymatrix{p}_{hp} - \widetilde{\mymatrix{p}}_{hp}) \big) + \psi_{hp}(\mymatrix{p}_{hp}) - \psi_{hp}(\widetilde{\mymatrix{p}}_{hp}) 
 \geq \ell(\myvec{u}_{hp} - \widetilde{\myvec{u}}_{hp}),
\end{align*}
together with the ellipticity of $a(\cdot,\cdot)$ give
\begin{align*}
 \alpha \, \big\Vert (\widetilde{\myvec{u}}_{hp} - \myvec{u}_{hp}), (\widetilde{\mymatrix{p}}_{hp} - \mymatrix{p}_{hp}) \big\Vert^2
 \leq a\big( (\widetilde{\myvec{u}}_{hp} - \myvec{u}_{hp}, \widetilde{\mymatrix{p}}_{hp} - \mymatrix{p}_{hp}), (\widetilde{\myvec{u}}_{hp} - \myvec{u}_{hp}, \widetilde{\mymatrix{p}}_{hp} - \mymatrix{p}_{hp}) \big)
 \leq 0,
\end{align*}
from which we deduce $(\myvec{u}_{hp},\mymatrix{p}_{hp}) = (\widetilde{\myvec{u}}_{hp},\widetilde{\mymatrix{p}}_{hp})$.
\end{proof}

We again note that under the assumption \eqref{eq:requirement_detF} it holds $\Lambda_{hp}^{(w)}=\Lambda_{hp}^{(s)}$ and, thus, \cite{Bammer2023Apriori} gives for uniform $h$ and uniform $p$
\begin{align*}
 \lim_{h/p \to 0} \big( \Vert \myvec{u}-\myvec{u}_{hp}\Vert_{1,\Omega}^2 + \Vert \mymatrix{p}-\mymatrix{p}_{hp}\Vert_{0,\Omega}^2 + \Vert\mymatrix{\lambda}-\mymatrix{\lambda}_{hp}\Vert_{0,\Omega}^2 \big) 
 = 0.
\end{align*} 
Furthermore, if $(\myvec{u},\mymatrix{p},\mymatrix{\lambda})\in H^s(\Omega,\myspace{R}^d) \times H^t(\Omega,\myspace{R}^{d\times d}) \times H^l(\Omega,\myspace{R}^{d\times d})$ with $s\geq 1$ and $t$, $l>d/2$ the a priori error estimates
\begin{align*}
 \Vert (\myvec{u}-\myvec{u}_{hp}, \mymatrix{p}-\mymatrix{p}_{hp}) \Vert^2  + \Vert\mymatrix{\lambda}-\mymatrix{\lambda}_{hp}\Vert_{0,\Omega}^2  
 \lesssim \frac{h^{\min(p,2s-2,t,l)}}{p^{\min(2s-2,t,l)}}
\end{align*}
are guaranteed. In the special case of lower-order finite elements (i.e.~if $p = 1$), we obtain the optimal order of convergence
\begin{align*}
 \Vert (\myvec{u}-\myvec{u}_{hp}, \mymatrix{p}-\mymatrix{p}_{hp}) \Vert^2  + \Vert\mymatrix{\lambda}-\mymatrix{\lambda}_{hp}\Vert_{0,\Omega}^2 
 \lesssim h^{2\min(1,s-1,t,l)};
\end{align*}
even if \eqref{eq:requirement_detF} does not hold true, cf.~\cite{Bammer2023Apriori}. Without the assumption \eqref{eq:requirement_detF} the convergence of higher-order finite element schemes based on \eqref{eq:discrete_variq_second_kind} is an open problem.

\medskip

% ---------------------------------------------------------------------------------------
%   SECTION: An abstract a posteriori error estimator
% ----------------------------------------------------------------------------------------

\section{A~Posteriori Error Estimates}
\label{sec:abstractAposteriori}

\subsection{Error estimates based on an auxiliary problem}\label{subsec:AuxProSec}

In this section we derive upper and lower error estimates, which are based on 
the introduction of an auxiliary problem that takes the form of a variational equation. We emphasize that the estimates are derived for an arbitrary triple 
\[
(\myvec{u}_N,\mymatrix{p}_N,\mymatrix{\lambda}_N)\in V\times Q\times Q, 
\]
which may be the solution of the discrete variational inequality \eqref{eq:discrete_variq_second_kind} or the solution of the discrete mixed formulation \eqref{eq:discrete_mixed_variationalF} (or some approximations of them resulting from iterative solution schemes). The auxiliary problem takes the form: Find a pair $(\myvec{u}^*,\mymatrix{p}^*)\in V\times Q$ such that
\begin{align}\label{eq:additional_problem}
 a\big( (\myvec{u}^*,\mymatrix{p}^*),(\myvec{v},\mymatrix{q}) \big) 
 = \ell(\myvec{v}) - (\mymatrix{\lambda}_N, \mymatrix{q})_{0,\Omega} \qquad 
 \forall \, (\myvec{v},\mymatrix{q})\in V\times Q.
\end{align}
Note that the unique existence of a solution $(\myvec{u}^*,\mymatrix{p}^*)$ of \eqref{eq:additional_problem} is guaranteed by the Lax-Milgram-Lemma. By subtracting \eqref{eq:mixed_variationalF_01} from \eqref{eq:additional_problem} we find the relation
\begin{align}\label{eq:proof_a_posteriori_01}
 a\big( (\myvec{u}^*-\myvec{u}, \mymatrix{p}^*-\mymatrix{p}), (\myvec{v},\mymatrix{q}) \big) + (\mymatrix{\lambda}_N - \mymatrix{\lambda}, \mymatrix{q})_{0,\Omega} 
 = 0 \qquad 
 \forall \, (\myvec{v},\mymatrix{q})\in V\times Q.
\end{align}
For $\mymatrix{\mu}\in\Lambda$ we introduce the global plasticity error contribution
\begin{align*}
 \E(\mymatrix{\mu}) 
 := \Vert \mymatrix{\mu}-\mymatrix{\lambda}_N\Vert_{0,\Omega}^2 + \psi(\mymatrix{p}_N) - (\mymatrix{\mu}, \mymatrix{p}_N)_{0,\Omega}
\end{align*}
and show the following upper error estimate.

\begin{theorem}\label{lem:a_posteriori_error}
For every $\mymatrix{\mu}\in\Lambda$ there holds
\begin{align}\label{eq:reliability_estimate}
 \Vert (\myvec{u}-\myvec{u}_N, \mymatrix{p}-\mymatrix{p}_N)\Vert^2 + \Vert \mymatrix{\lambda}-\mymatrix{\lambda}_N\Vert_{0,\Omega}^2 
 \lesssim \Vert (\myvec{u}^*-\myvec{u}_N, \mymatrix{p}^*-\mymatrix{p}_N)\Vert^2 + \E(\mymatrix{\mu}).
\end{align}
\end{theorem}
\begin{proof}
Choosing $(\myvec{u}-\myvec{u}_N, \mymatrix{p}-\mymatrix{p}_N)\in V\times Q$ as test function in \eqref{eq:proof_a_posteriori_01}, taking the ellipticity and continuity of $a(\cdot,\cdot)$ into account and applying Young's inequality yield for any $\varepsilon>0$
\begin{align*}
 \alpha \, \Vert (\myvec{u}-\myvec{u}_N, \mymatrix{p}-\mymatrix{p}_N) \Vert^2
 &\leq a\big( (\myvec{u}-\myvec{u}_N, \mymatrix{p}-\mymatrix{p}_N), (\myvec{u}-\myvec{u}_N, \mymatrix{p}-\mymatrix{p}_N) \big) \\
 &= a\big( (\myvec{u}^*-\myvec{u}_N, \mymatrix{p}^*-\mymatrix{p}_N), (\myvec{u}-\myvec{u}_N, \mymatrix{p}-\mymatrix{p}_N) \big) + (\mymatrix{\lambda}_N - \mymatrix{\lambda}, \mymatrix{p}-\mymatrix{p}_N)_{0,\Omega} \\
 &\leq c_a \, \Vert (\myvec{u}^*-\myvec{u}_N, \mymatrix{p}^*-\mymatrix{p}_N) \Vert \, \Vert (\myvec{u}-\myvec{u}_N, \mymatrix{p}-\mymatrix{p}_N) \Vert + (\mymatrix{\lambda}_N - \mymatrix{\lambda}, \mymatrix{p}-\mymatrix{p}_N)_{0,\Omega} \\
 & \leq \frac{c_a^2}{4\varepsilon} \, \Vert (\myvec{u}^*-\myvec{u}_N, \mymatrix{p}^*-\mymatrix{p}_N) \Vert^2 + \varepsilon \, \Vert (\myvec{u}-\myvec{u}_N, \mymatrix{p}-\mymatrix{p}_N) \Vert^2 + (\mymatrix{\lambda}_N - \mymatrix{\lambda}, \mymatrix{p}-\mymatrix{p}_N)_{0,\Omega}.
\end{align*}
Exploiting \eqref{eq:mixed_variationalF_02} and $\mymatrix{\lambda} \in \Lambda$, we obtain by Cauchy-Schwarz's and Young's inequality for $\mymatrix{\mu}\in\Lambda$
\begin{align*}
 (\mymatrix{\lambda}_N-\mymatrix{\lambda}, \mymatrix{p}-\mymatrix{p}_N)_{0,\Omega}
 &= ( \mymatrix{\lambda}_N-\mymatrix{\mu}, \mymatrix{p}-\mymatrix{p}_N)_{0,\Omega} + ( \mymatrix{\mu}-\mymatrix{\lambda}, \mymatrix{p}-\mymatrix{p}_N )_{0,\Omega} \\
 &=( \mymatrix{\lambda}_N-\mymatrix{\mu}, \mymatrix{p}-\mymatrix{p}_N)_{0,\Omega} + ( \mymatrix{\mu}-\mymatrix{\lambda}, \mymatrix{p})_{0,\Omega} + (\mymatrix{\lambda}, \mymatrix{p}_N )_{0,\Omega} - ( \mymatrix{\mu}, \mymatrix{p}_N )_{0,\Omega} \\
 & \leq ( \mymatrix{\lambda}_N-\mymatrix{\mu}, \mymatrix{p}-\mymatrix{p}_N)_{0,\Omega} + (\mymatrix{\lambda}, \mymatrix{p}_N )_{0,\Omega} - ( \mymatrix{\mu}, \mymatrix{p}_N )_{0,\Omega} \\
 & \leq ( \mymatrix{\lambda}_N-\mymatrix{\mu}, \mymatrix{p}-\mymatrix{p}_N)_{0,\Omega} + \psi(\mymatrix{p}_N) - ( \mymatrix{\mu}, \mymatrix{p}_N )_{0,\Omega} \\
 & \leq  \frac{1}{4\varepsilon} \, \Vert \mymatrix{\mu}-\mymatrix{\lambda}_N \Vert_{0,\Omega}^2 + \varepsilon \, \Vert \mymatrix{p}-\mymatrix{p}_N\Vert_{0,\Omega}^2 +  \psi(\mymatrix{p}_N) - ( \mymatrix{\mu}, \mymatrix{p}_N )_{0,\Omega} .
\end{align*}
Combining the last two estimates yields
\begin{align*}
 (\alpha - 2\varepsilon) \, \Vert (\myvec{u}-\myvec{u}_N, \mymatrix{p}-\mymatrix{p}_N) \Vert^2
 \leq \frac{c_a^2}{4\varepsilon} \, \Vert (\myvec{u}^*-\myvec{u}_N, \mymatrix{p}^*-\mymatrix{p}_N) \Vert^2 + \frac{1}{4\varepsilon} \, \Vert \mymatrix{\mu}-\mymatrix{\lambda}_N \Vert_{0,\Omega}^2 + \psi(\mymatrix{p}_N) - ( \mymatrix{\mu}, \mymatrix{p}_N )_{0,\Omega}.
\end{align*}
For $\varepsilon < \alpha / 2$ we therefore deduce
\begin{align}\label{eq:proof_a_posteriori_01b}
  \Vert (\myvec{u}-\myvec{u}_N, \mymatrix{p}-\mymatrix{p}_N) \Vert^2
& \leq \frac{c_a^2}{c_{\varepsilon}} \, \Vert (\myvec{u}^*-\myvec{u}_N, \mymatrix{p}^*-\mymatrix{p}_N) \Vert^2  + \frac{1}{c_{\varepsilon}} \, \Vert \mymatrix{\mu}-\mymatrix{\lambda}_N \Vert_{0,\Omega}^2 
+ \frac{4\varepsilon}{c_{\varepsilon}} \, \big( \psi(\mymatrix{p}_N) - ( \mymatrix{\mu}, \mymatrix{p}_N )_{0,\Omega} \big)
\end{align}
with $c_{\varepsilon} := 4\varepsilon(\alpha - 2\varepsilon)$. As $\mymatrix{\lambda}, \mymatrix{\lambda}_N\in Q$ we may choose $\myvec{v}=\myvec{o}$ and $\mymatrix{q}=\mymatrix{\lambda}-\mymatrix{\lambda}_N$ in \eqref{eq:proof_a_posteriori_01} to obtain 
\begin{align*}
 \Vert\mymatrix{\lambda}-\mymatrix{\lambda}_N\Vert_{0,\Omega}^2
 &= (\mymatrix{\lambda}-\mymatrix{\lambda}_N, \mymatrix{\lambda}-\mymatrix{\lambda}_N)_{0,\Omega}  \\
 &= - a\big( (\myvec{u}-\myvec{u}^*, \mymatrix{p}-\mymatrix{p}^*), (\myvec{o}, \mymatrix{\lambda}-\mymatrix{\lambda}_N) \big) \\
 &\leq c_a \, \Vert (\myvec{u}-\myvec{u}^*, \mymatrix{p}-\mymatrix{p}^*) \Vert \, \Vert \mymatrix{\lambda}-\mymatrix{\lambda}_N\Vert_{0,\Omega},
\end{align*}
where the last inequalitiy results from the continuity of $a(\cdot,\cdot)$. Thus, the triangle inequality gives
\begin{align*}
 \Vert\mymatrix{\lambda}-\mymatrix{\lambda}_N\Vert_{0,\Omega}
 \leq c_a \, \Vert (\myvec{u}-\myvec{u}^*, \mymatrix{p}-\mymatrix{p}^*) \Vert 
 \leq c_a \, \Vert (\myvec{u}-\myvec{u}_N, \mymatrix{p}-\mymatrix{p}_N) \Vert + c_a \, \Vert (\myvec{u}^*-\myvec{u}_N, \mymatrix{p}^*-\mymatrix{p}_N) \Vert.
\end{align*}
Hence, we have
\begin{align*}
 \Vert (\myvec{u}-\myvec{u}_N, \mymatrix{p}-\mymatrix{p}_N) \Vert^2 + \Vert\mymatrix{\lambda}-\mymatrix{\lambda}_N\Vert_{0,\Omega}^2 & \leq (1 +2c_a^2) \, \Vert (\myvec{u}-\myvec{u}_N, \mymatrix{p}-\mymatrix{p}_N) \Vert^2 + 2c_a^2 \, \Vert (\myvec{u}^*-\myvec{u}_N, \mymatrix{p}^*-\mymatrix{p}_N) \Vert^2,
\end{align*} 
from which we deduce, by inserting \eqref{eq:proof_a_posteriori_01b},
\begin{align*}
  \Vert (\myvec{u}-\myvec{u}_N, \mymatrix{p}-\mymatrix{p}_N) \Vert^2 + \Vert\mymatrix{\lambda}-\mymatrix{\lambda}_N\Vert_{0,\Omega}^2 
  &\leq \frac{c_a^2 \big( 1 + 2(c_a^2 + c_\varepsilon) \big)}{c_\varepsilon} \, \Vert (\myvec{u}^*-\myvec{u}_N, \mymatrix{p}^*-\mymatrix{p}_N) \Vert^2 + \frac{1+c_a^2}{c_\varepsilon} \, \Vert \mymatrix{\mu}-\mymatrix{\lambda}_N \Vert_{0,\Omega}^2 \\
  & \qquad + \frac{4\varepsilon(1+2c_a^2)}{c_\varepsilon} \, \big( \psi(\mymatrix{p}_N) - ( \mymatrix{\mu}, \mymatrix{p}_N )_{0,\Omega} \big).  
\end{align*}
Since $\mymatrix{\mu}\in\Lambda$ and $\mymatrix{p}_N \in Q$ we have $\psi(\mymatrix{p}_N) - ( \mymatrix{\mu}, \mymatrix{p}_N ) \geq 0$. Thus, the assertion follows by bounding the three constants in the above estimate by their common maximum.
\end{proof}

To obtain lower error estimates we minimize $\E(\cdot)$ over the set $\Lambda$. For this purpose, we define 
\begin{align}\label{eq:defi_minimizer_mu}
 \mymatrix{\mu}^* 
 := \min\left\lbrace 1, \frac{\sigma_y}{ \; \big\vert \widehat{\mymatrix{\mu}} \big\vert_F} \right\rbrace \widehat{\mymatrix{\mu}}, \qquad 
 \widehat{\mymatrix{\mu}}:= \mymatrix{\lambda}_{N} + \frac{1}{2} \, \mymatrix{p}_{N}.
\end{align}
Obviously, we have $\mymatrix{\mu}^*\in\Lambda$ and as it is the projection of $\widehat{\mymatrix{\mu}}$ with respect to $|\cdot|_F$ we obtain
\begin{align*}
 \vert \widehat{\mymatrix{\mu}}  - \mymatrix{\mu}^* \vert_F
 \leq \vert \widehat{\mymatrix{\mu}}  - \mymatrix{\mu} \vert_F \quad
 \text{a.e.~in } \Omega \qquad 
 \forall \, \mymatrix{\mu}\in \Lambda,
\end{align*}
which implies
\begin{align*}
 \Vert \widehat{\mymatrix{\mu}} - \mymatrix{\mu}^* \Vert_{0,\Omega} 
 \leq \Vert \widehat{\mymatrix{\mu}}  - \mymatrix{\mu} \Vert_{0,\Omega} \qquad
 \forall \, \mymatrix{\mu}\in \Lambda,
\end{align*}
i.e.~$\mymatrix{\mu}^*$ represents the $L^2$-projection of $\widehat{\mymatrix{\mu}}$ onto $\Lambda$. Applying Lemma~\ref{prop:abstract_minP} from the appendix we conclude that $\mymatrix{\mu}^*$ uniquely minimizes $\E(\cdot)$ over the set $\Lambda$, i.e.
\pagebreak
\begin{align}\label{eq:local_minimization_apost}
 \E(\mymatrix{\mu}^*) = \min_{\mymatrix{\mu}\in\Lambda} \E(\mymatrix{\mu}).
 %= \min_{\mymatrix{\mu}\in C_T} \Vert \mymatrix{\mu} - \mymatrix{\lambda}_N \Vert_{0,T}^2 - ( \mymatrix{\mu}, \mymatrix{p}_N )_{0,T}.
\end{align}

\begin{lemma}\label{lem:optimal_lambda}
There holds 
\begin{align}\label{eq:efficiency_E}
 \E(\mymatrix{\mu}^*) 
 & \leq \Vert \mymatrix{\lambda}-\mymatrix{\lambda}_N\Vert_{0,\Omega}^2 + \big( \Vert \sigma_y \Vert_{0,\Omega} + \Vert \mymatrix{\lambda}\Vert_{0,\Omega} \big) \, \Vert \mymatrix{p}-\mymatrix{p}_N\Vert_{0,\Omega}.
\end{align}
\end{lemma}

\begin{proof}
The identity $\mymatrix{\lambda} : \mymatrix{p} = \sigma_y \, \vert\mymatrix{p}\vert_F$ a.e.~in $\Omega$, see \eqref{eq:plambda}, the minimality of $\mymatrix{\mu}^{\ast}$ and the Cauchy-Schwarz inequality yield
\begin{align*}
 \E(\mymatrix{\mu}^*) 
 &\leq \E(\mymatrix{\lambda}) \\
 &= \Vert \mymatrix{\lambda}-\mymatrix{\lambda}_N\Vert_{0,\Omega}^2 + (\sigma_y, |\mymatrix{p}_N|_F)_{0,\Omega} - (\mymatrix{\lambda}, \mymatrix{p}_N)_{0,\Omega} \\
 &= \Vert \mymatrix{\lambda}-\mymatrix{\lambda}_N\Vert_{0,\Omega}^2 + \big( \sigma_y, \vert \mymatrix{p}_N\vert_F -\vert \mymatrix{p} \vert_F \big)_{0,\Omega} + ( \mymatrix{\lambda},  \mymatrix{p} - \mymatrix{p}_N )_{0,\Omega} \\
& \leq \Vert \mymatrix{\lambda}-\mymatrix{\lambda}_N\Vert_{0,\Omega}^2 + \big( \Vert \sigma_y \Vert_{0,\Omega} + \Vert \mymatrix{\lambda}\Vert_{0,\Omega} \big) \, \Vert \mymatrix{p}-\mymatrix{p}_N\Vert_{0,\Omega},
\end{align*}
where the last estimate results from the application of the reverse triangle inequality.
\end{proof}

Note that as $\mymatrix{\lambda}\in\Lambda$ we have $\Vert \mymatrix{\lambda}\Vert_{0,\Omega} \leq \Vert \sigma_y \Vert_{0,\Omega}$.

\begin{lemma}\label{lem:effLemma}
There holds
\begin{align*}
 \Vert (\myvec{u}^*-\myvec{u}_N, \mymatrix{p}^*-\mymatrix{p}_N) \Vert^2 
 \lesssim \Vert (\myvec{u}-\myvec{u}_N, \mymatrix{p}-\mymatrix{p}_N) \Vert^2 + \Vert \mymatrix{\lambda}-\mymatrix{\lambda}_N\Vert_{0,\Omega}^2.
\end{align*}
\end{lemma}

\begin{proof}
Testing \eqref{eq:proof_a_posteriori_01} by $(\myvec{u}^*-\myvec{u}, \mymatrix{p}^* -\mymatrix{p})\in V\times Q$ yields
\begin{align*}
 a\big( (\myvec{u}^*-\myvec{u}, \mymatrix{p}^*-\mymatrix{p}), (\myvec{u}^*-\myvec{u}, \mymatrix{p}^*-\mymatrix{p}) \big) 
 = (\mymatrix{\lambda}-\mymatrix{\lambda}_N, \mymatrix{p}^* - \mymatrix{p})_{0,\Omega}.
\end{align*}
Hence, the ellipticity of $a(\cdot,\cdot)$ and the Cauchy-Schwarz inequality imply
\begin{align*}
 \alpha \, \Vert (\myvec{u}^*-\myvec{u}, \mymatrix{p}^*-\mymatrix{p})\Vert^2 
 \leq \Vert \mymatrix{\lambda}-\mymatrix{\lambda}_N\Vert_{0,\Omega} \, \Vert \mymatrix{p}^*-\mymatrix{p}\Vert_{0,\Omega}
 \leq \Vert \mymatrix{\lambda}-\mymatrix{\lambda}_N\Vert_{0,\Omega} \, \Vert (\myvec{u}^*-\myvec{u}, \mymatrix{p}^*-\mymatrix{p})\Vert.
\end{align*}
Finally, by the triangle inequality, we have
\begin{align*} 
 \Vert (\myvec{u}^*-\myvec{u}_N, \mymatrix{p}^*-\mymatrix{p}_N)\Vert 
 & \leq \Vert (\myvec{u}^*-\myvec{u}, \mymatrix{p}^*-\mymatrix{p})\Vert + \Vert (\myvec{u}-\myvec{u}_N, \mymatrix{p}-\mymatrix{p}_N)\Vert \\
 & \leq \frac{1}{\alpha} \, \Vert \mymatrix{\lambda}-\mymatrix{\lambda}_N\Vert_{0,\Omega} + \Vert (\myvec{u}-\myvec{u}_N, \mymatrix{p}-\mymatrix{p}_N)\Vert,
\end{align*}
which completes the argument.
\end{proof}

\begin{theorem}\label{thm:efficiency}
There holds
\begin{align}\label{eq:efficiency_est}
 \Vert (\myvec{u}^*-\myvec{u}_N, \mymatrix{p}^*-\mymatrix{p}_N)\Vert^2 + \E(\mymatrix{\mu}^*)\lesssim \Vert (\myvec{u}-\myvec{u}_N, \mymatrix{p}-\mymatrix{p}_N)\Vert^2 + \Vert \mymatrix{\lambda}-\mymatrix{\lambda}_N\Vert_{0,\Omega}^2 +  \Vert \mymatrix{p}-\mymatrix{p}_N\Vert_{0,\Omega}.
\end{align}
\end{theorem}
\begin{proof}
 The assertion follows directly by combining Lemma~\ref{lem:optimal_lambda} and Lemma~\ref{lem:effLemma}.
\end{proof}

% Let $\eta_{\veq}\geq 0$ such that there exists a constant $c_{\text{rel}} > 0$ with
% \begin{align*}
%  \Vert (\mymatrix{u}^* - \mymatrix{u}_N, \mymatrix{p}^* - \mymatrix{p}_N ) \Vert^2
%  \leq c_{\text{rel}} \, (\eta_{\veq}^2 + \osc^2)
% \end{align*} 
% for some (data oscillation) $\osc\geq 0$. Additionally, let $\eta_{\veq}\geq 0$ fulfill
% \begin{align*}
%  c_{\text{eff}} \, \eta_{\veq}^2 
%  \leq \Vert (\mymatrix{u}^* - \mymatrix{u}_N, \mymatrix{p}^* - \mymatrix{p}_N ) \Vert^2 + \osc^2
% \end{align*}
% for some constant $c_{\text{eff}} > 0$ and define
% \begin{align*}
%  \eta_{\viq}^2(\mymatrix{\mu}) := \eta_{\veq}^2 + E^2(\mymatrix{\mu}),\quad \mymatrix{\mu}\in\Lambda.
% \end{align*}
% We immediately obtain from Lemma~\ref{lem:a_posteriori_error}, Lemma~\ref{lem:optimal_lambda} and Lemma~\ref{thm:efficiency} the following assertion.
% 
% \begin{theorem}
%  It holds the reliability estimate
%  \begin{align*}
%  \Vert (\myvec{u}-\myvec{u}_N, \mymatrix{p}-\mymatrix{p}_N)\Vert^2 + \Vert \mymatrix{\lambda}-\mymatrix{\lambda}_N\Vert_{0,\Omega}^2
%  \lesssim \eta_{\viq}^2(\mymatrix{\mu}) + \osc^2 \quad
%  \forall \, \mymatrix{\mu}\in\Lambda.
% \end{align*}
% and the efficiency estimate
% \begin{align}\label{eq:abstract_efficiencyEst}
%  \eta_{\viq}(\mymatrix{\mu}^*)
%  \lesssim \Vert (\myvec{u}-\myvec{u}_N, \mymatrix{p}-\mymatrix{p}_N)\Vert^2 + \Vert \mymatrix{\lambda}-\mymatrix{\lambda}_N\Vert_{0,\Omega}^2 +  \Vert \mymatrix{p}-\mymatrix{p}_N\Vert_{0,\Omega} + \osc^2.
% \end{align}
% \end{theorem}

The lower estimate in \eqref{eq:efficiency_est} is suboptimal as it contains the linear term $\Vert \mymatrix{p}-\mymatrix{p}_N\Vert_{0,\Omega}$ resulting from the estimate \eqref{eq:efficiency_E}. Under some additional assumptions on $\mymatrix{\lambda}_N$ and $
\mymatrix{q}_N$ this linear term as well as $\E(\mymatrix{\cdot})$ do not occur in the upper and lower estimate.
\begin{corollary}\label{cor:simple}
Let $\mymatrix{\lambda}_N\in\Lambda$ and let the relation
\begin{align}\label{eq:psi_rel}
 (\mymatrix{\lambda}_N,\mymatrix{q}_N)_{0,\Omega}
 = \psi(\mymatrix{q}_N)
\end{align}
be fulfilled. Then,
\begin{align*}
 \Vert (\myvec{u}^*-\myvec{u}_N, \mymatrix{p}^*-\mymatrix{p}_N)\Vert^2 
 \approx \Vert (\myvec{u}-\myvec{u}_N, \mymatrix{p}-\mymatrix{p}_N)\Vert^2 + \Vert \mymatrix{\lambda}-\mymatrix{\lambda}_N\Vert_{0,\Omega}^2.
\end{align*}
\end{corollary}
\begin{proof}
We have $\E(\mymatrix{\lambda}_N)=0$. Thus, Theorem~\ref{lem:a_posteriori_error} and Lemma~\ref{lem:effLemma} yield the assertion.
\end{proof}

\begin{remark*}
The relation~\eqref{eq:psi_rel} holds true if $\mymatrix{p}_{N}\in Q_N$ for a subspace $Q_N\subset Q$ and
\begin{align*}
 (\mymatrix{\mu}_{N}-\mymatrix{\lambda}_{N}, \mymatrix{p}_{N})_{0,\Omega}
 \leq 0 \qquad 
 \forall \, \mymatrix{\mu}_{N} \in \Lambda_N
\end{align*}
with $\Lambda_N:=\Lambda\cap Q_N$. This in particular holds true for the discretizations in Section~\ref{subsec:two_discr_mixed} with $p_T=1$ for all $T\in\mathcal{T}_h$. We refer to \cite{schroder2011error}, where some lower-order finite element discretization spaces are proposed, for a mixed discrete formulation satisfying these assumptions. Indeed, the a~posteriori error estimates introduced in \cite{schroder2011error} coincide with the estimates of Corollary~\ref{cor:simple}.
\end{remark*}

%\medskip

% ----------------------------------------------------------------------------------------

\subsection{A residual based a~posteriori error estimator}
\label{sec:residualErrorEst}

In this section we apply Theorem~\ref{lem:a_posteriori_error} and Theorem~\ref{thm:efficiency} in order to derive a residual-based a postoriori error estimator. For this purpose, let $\myvec{u}_N\in V_N:=V_{hp}$ and $\mymatrix{p}_N\in Q_N:=Q_{hp}$. Moreover, take $\mymatrix{\lambda}_N\in Q$ such that
\begin{align}\label{eq:find_lambda}
  ( \mymatrix{\lambda}_N, \mymatrix{q}_N )_{0,\Omega}
 = \ell(\myvec{v}_N) -  a\big( ( \myvec{u}_N, \mymatrix{p}_N ) , (\myvec{v}_N,\mymatrix{q}_N) \big) \qquad  
 \forall \, (\myvec{v}_N,\mymatrix{q}_N) \in V_N\times Q_N.
\end{align}
Choosing $\myvec{v}_N = \myvec{0}$ in \eqref{eq:find_lambda} we conclude for arbitrary $\mymatrix{q}_N\in Q_N$ 
\begin{align*}
 ( \mymatrix{\lambda}_N, \mymatrix{q}_N )_{0,\Omega} 
 & = -a\big( (\myvec{u}_N, \mymatrix{p}_N) , (\myvec{o}, \mymatrix{q}_N) \big)
 = ( \mymatrix{\sigma}(\myvec{u}_N,\mymatrix{p}_N) - \myspace{H} \, \mymatrix{p}_N, \mymatrix{q}_N )_{0,\Omega} 
 = \big( \operatorname{dev}(\mymatrix{\sigma}\big(\myvec{u}_N,\mymatrix{p}_N) - \myspace{H} \, \mymatrix{p}_N\big), \mymatrix{q}_N \big)_{0,\Omega},
\end{align*}
where we use the definition of $a(\cdot,\cdot)$ and the identity \eqref{eq:identity_deviatoricPart}.
Thus, we may simply choose
\begin{align*}
 \mymatrix{\lambda}_N:= \operatorname{dev}\big( \mymatrix{\sigma}(\myvec{u}_N,\mymatrix{p}_N) - \myspace{H} \, \mymatrix{p}_N \big).
\end{align*}
Alternatively, if one has a solution $(\myvec{u}_N,\mymatrix{p}_N,\mymatrix{\lambda}_N)\in V_N\times Q_N\times Q_N$ of a discretization of the mixed problem \eqref{eq:mixed_variationalF} at hand (as, for instance, in \eqref{eq:discrete_mixed_variationalF}) one can directly use this discrete $\mymatrix{\lambda}_N$.

\begin{lemma}\label{lem:galerkinO}
 For an arbitrary $\myvec{v}_N\in V_N$ it holds
 \begin{align*}
 \Vert ( \myvec{u}^*-\myvec{u}_N, \mymatrix{p}^*-\mymatrix{p}_N ) \Vert^2
 &\lesssim \big| \ell ( \myvec{u}^*-\myvec{u}_N-\myvec{v}_N ) 
 - \big( \mymatrix{\sigma}( \myvec{u}_N, \mymatrix{p}_N ), \mymatrix{\varepsilon}( \myvec{u}^* - \myvec{u}_N - \myvec{v}_N ) \big)_{0,\Omega} \big| \\
 & \qquad + \big\Vert \operatorname{dev}\big(\mymatrix{\sigma} ( \myvec{u}_N,\mymatrix{p}_N ) - \myspace{H} \, \mymatrix{p}_N \big) - \mymatrix{\lambda}_N \big\Vert_{0,\Omega}^2.
\end{align*}
\end{lemma}

\begin{proof}
Subtracting \eqref{eq:find_lambda} from the auxiliary problem \eqref{eq:additional_problem} (choosing $(\myvec{v}_N,\mymatrix{0})$ as test function in both cases) yields
\begin{align} \label{eq:GalOrtho}
 a\big( ( \myvec{u}^*-\myvec{u}_N, \mymatrix{p}^*-\mymatrix{p}_N ), (\myvec{v}_N, \mymatrix{0}) \big) 
 = 0.
\end{align}
Hence, exploiting the ellipticity of $a(\cdot,\cdot)$ gives
\begin{align*}
 \alpha \, \Vert ( \myvec{u}^*-\myvec{u}_N, \mymatrix{p}^*-\mymatrix{p}_N ) \Vert^2
 &\leq a\big( ( \myvec{u}^*-\myvec{u}_N, \mymatrix{p}^*-\mymatrix{p}_N ), ( \myvec{u}^*-\myvec{u}_N, \mymatrix{p}^*-\mymatrix{p}_N) \big) \\
 &= a\big( ( \myvec{u}^*-\myvec{u}_N, \mymatrix{p}^*-\mymatrix{p}_N ), ( \myvec{u}^*-\myvec{u}_N-\myvec{v}_N, \mymatrix{p}^*-\mymatrix{p}_N ) \big) \\
 &= a\big( ( \myvec{u}^*, \mymatrix{p}^* ), ( \myvec{u}^*-\myvec{u}_N-\myvec{v}_N, \mymatrix{p}^*-\mymatrix{p}_N) \big) - a\big( ( \myvec{u}_N, \mymatrix{p}_N ), ( \myvec{u}^*-\myvec{u}_N-\myvec{v}_N, \mymatrix{p}^*-\mymatrix{p}_N) \big),
\end{align*}
from which we deduce
\begin{align*}
 &\alpha \, \Vert ( \myvec{u}^*-\myvec{u}_N, \mymatrix{p}^*-\mymatrix{p}_N ) \Vert^2 \\
 & \qquad \leq \ell ( \myvec{u}^*-\myvec{u}_N-\myvec{v}_N )
 - \big( \mymatrix{\sigma}( \myvec{u}_N, \mymatrix{p}_N ), \mymatrix{\varepsilon}( \myvec{u}^* - \myvec{u}_N - \myvec{v}_N ) \big)_{0,\Omega} + \left(\operatorname{dev}\big(\mymatrix{\sigma} ( \myvec{u}_N,\mymatrix{p}_N ) - \myspace{H} \, \mymatrix{p}_N \big) - \mymatrix{\lambda}_N, \mymatrix{p}^*-\mymatrix{p}_N\right)_{0,\Omega}
\end{align*}
due to \eqref{eq:additional_problem}, the definition of $a(\cdot,\cdot)$ and the identity \eqref{eq:identity_deviatoricPart}. Hence, by using the inequality from Cauchy-Schwarz's and Young's inequality with some $0<\varepsilon<\alpha$ we obtain
\begin{align*}
\alpha \, \Vert ( \myvec{u}^*-\myvec{u}_N, \mymatrix{p}^*-\mymatrix{p}_N ) \Vert^2
 &\leq \ell ( \myvec{u}^*-\myvec{u}_N-\myvec{v}_N ) 
 - \big( \mymatrix{\sigma}( \myvec{u}_N, \mymatrix{p}_N ), \mymatrix{\varepsilon}( \myvec{u}^* - \myvec{u}_N - \myvec{v}_N ) \big)_{0,\Omega} \\
 & \qquad + \frac{1}{4\varepsilon} \, \big\Vert \operatorname{dev}\big(\mymatrix{\sigma} ( \myvec{u}_N,\mymatrix{p}_N ) - \myspace{H} \, \mymatrix{p}_N \big) - \mymatrix{\lambda}_N \big\Vert_{0,\Omega}^2 + \varepsilon \, \big\Vert ( \myvec{u}^*-\myvec{u}_N, \mymatrix{p}^*-\mymatrix{p}_N ) \big\Vert^2.
\end{align*}
\vspace*{0.5cm}
This finally gives
\begin{align*}
 \Vert ( \myvec{u}^*-\myvec{u}_N, \mymatrix{p}^*-\mymatrix{p}_N ) \Vert^2 
 &\leq \frac{1}{\alpha-\varepsilon} \left(\ell ( \myvec{u}^*-\myvec{u}_N-\myvec{v}_N ) 
 - \big( \mymatrix{\sigma}( \myvec{u}_N, \mymatrix{p}_N ), \mymatrix{\varepsilon}( \myvec{u}^* - \myvec{u}_N - \myvec{v}_N ) \big)_{0,\Omega}\right) \\
 & \qquad + \frac{1}{4\varepsilon(\alpha-\varepsilon)} \, \big\Vert \operatorname{dev}\big(\mymatrix{\sigma} ( \myvec{u}_N,\mymatrix{p}_N ) - \myspace{H} \, \mymatrix{p}_N \big) - \mymatrix{\lambda}_N \big\Vert_{0,\Omega}^2,
\end{align*}
which completes the argument.
\end{proof}

In the following we restrict ourselves to the two-dimensional case, i.e.~$d=2$ (as we refer to arguments known from a~posteriori error control specifically for this case in Theorem~\ref{thm:residualErrorEst} and Lemma~\ref{lem:localEta}). For each element $T\in\mathcal{T}_h$ we denote the sets of edges of $T$ which lie in the interior of $\Omega$ and on the Neumann-boundary $\Gamma_N$ by $\mathcal{E}_T^{I}$ and $\mathcal{E}_T^N$, respectively. In addition to the local element size $h_T$ and polynomial degree $p_T$ of an element $T\in\mathcal{T}_h$, let us denote the local edge size and the polynomial degree of its edges $e\in \mathcal{E}_T^{I} \cup \mathcal{E}_T^{N}$ by $h_e$ and $p_e$, respectively.
Moreover, let $\myvec{n}_e$ be a unit normal of the edge $e$, which coincides with the outer unit normal $\myvec{n}$ on the Neumann-boundary $\Gamma_N$. Finally, let $\llbracket \cdot \rrbracket$ denote the usual jump function. Herewith, we introduce the local error contributions for $T\in\mathcal{T}_h$
\begin{align}\label{eq:eta_T}
 \eta_T^2 
 := \frac{h_T^2}{p_T^2} \, \big\Vert \myvec{f}_N + \operatorname{div}\mymatrix{\sigma}( \myvec{u}_N,\mymatrix{p}_N ) \big\Vert_{0,T}^2 
 + \sum_{e \in \mathcal{E}_T^{I}}\frac{h_e}{2p_e} \, \big\Vert  \big\llbracket \mymatrix{\sigma}( \myvec{u}_N,\mymatrix{p}_N ) \, \myvec{n}_e \big\rrbracket \big\Vert_{0,e}^2 
 + \sum_{e \in \mathcal{E}_T^{N}}  \frac{h_e}{p_e} \, \big\Vert \mymatrix{\sigma} ( \myvec{u}_N,\mymatrix{p}_N ) \, \myvec{n}_e - \myvec{g}_N \big\Vert_{0,e}^2 
\end{align}
as well as the data oscillation terms
\begin{align*}
 \osc_T^2 
 := \frac{h_T^2}{p_T^2} \, \Vert \myvec{f} - \myvec{f}_N  \Vert_{0,T}^2 
 + \sum_{e \in \mathcal{E}_T^{N}}  \frac{h_e}{p_e} \, \Vert \myvec{g} - \myvec{g}_N\Vert_{0,e}^2,
\end{align*}
where $\myvec{f}_N \in [W_{hp}]^d$ and $\myvec{g} \in [W_{hp}|_{\Gamma_N}]^d$ represent the $L^2$-projection of $\myvec{f}\in L^2(\Omega,\mathbb{R}^d)$ and $\myvec{g}\in L^2(\Gamma_N,\mathbb{R}^d)$, respectively. For $\mymatrix{\mu}\in\Lambda$ we write
\begin{align}\label{eq:eta_mu}
 \eta^2(\mymatrix{\mu}) := \sum_{T \in \mathcal{T}_h} \eta^2_T(\mymatrix{\mu}), \qquad
 \osc^2 :=\sum_{T \in \mathcal{T}_h}  \osc_T^2
\end{align}
with the local error contributions
\begin{align*}
 \eta^2_T(\mymatrix{\mu}) := \eta_T^2+ \big\Vert \operatorname{dev} \big( \mymatrix{\sigma} (\myvec{u}_N,\mymatrix{p}_N ) - \myspace{H} \mymatrix{p}_N \big) - \mymatrix{\lambda}_N \big\Vert_{T,\Omega}^2 
 + \E_T(\mymatrix{\mu})
\end{align*}
and the local plasticity error contributions
 \begin{align}\label{eq:apost_localErrorPart}
  \E_T(\mymatrix{\mu}) 
  := \Vert \mymatrix{\mu} - \mymatrix{\lambda}_N\Vert_{0,T}^2 + (\sigma_y, \vert \mymatrix{p}_N\vert_F )_{0,T} - (\mymatrix{\mu} , \mymatrix{p}_N)_{0,T}, 
  \qquad T\in\mathcal{T}.
 \end{align}
Obviously, there holds
\begin{align*}
 \E(\mymatrix{\mu}) = \sum_{T\in\mathcal{T}_h} \E_T(\mymatrix{\mu}). 
\end{align*}
Herewith, we obtain the following reliability estimate for the error estimator $\eta^2(\mymatrix{\mu})$:
\begin{theorem} \label{thm:residualErrorEst}
For every $\mymatrix{\mu}\in\Lambda$ there holds
\begin{align} \label{eq:residual_est_reliable}
 \Vert ( \myvec{u}-\myvec{u}_N, \mymatrix{p}-\mymatrix{p}_N ) \Vert^2 + \Vert \mymatrix{\lambda}-\mymatrix{\lambda}_N \Vert_{0,\Omega}^2  
 &\lesssim  \eta^2(\mymatrix{\mu}) + \osc^2.
\end{align}
\end{theorem}
\begin{proof}
Choosing $\myvec{v}_N$ to be the Clement interpolation of $\myvec{u}^*-\myvec{u}_N$ (as, for instance, introduced in \cite{Melenk2005}) we conclude from Lemma~\ref{lem:galerkinO} and the usual arguments in the derivation of reliable residual-based a~posteriori error estimators (applying elementwise integration by parts and using Cauchy-Schwarz's inequality, see, for instance, \cite[Prop.~4.1]{Melenk2005} for the details)
\pagebreak
\begin{align*}
 \Vert \big( \myvec{u}^*-\myvec{u}_N, \mymatrix{p}^*-\mymatrix{p}_N \big) \Vert^2 
 & \lesssim \sum_{T \in \mathcal{T}_h} \frac{h_T^2}{p_T^2} \, \big\Vert \myvec{f} + \operatorname{div}\mymatrix{\sigma} ( \myvec{u}_N,\mymatrix{p}_N ) \big\Vert_{0,T}^2 + \sum_{e\in \mathcal{E}_T^{I}} \frac{h_e}{2p_e} \, \big\Vert \big\llbracket \mymatrix{\sigma} ( \myvec{u}_N,\mymatrix{p}_N ) \, \myvec{n}_e \big\rrbracket \big\Vert_{0,e}^2 \\
 & \qquad + \sum_{e\in \mathcal{E}_T^N} \frac{h_e}{p_e} \, \big\Vert \mymatrix{\sigma}( \myvec{u}_N,\mymatrix{p}_N ) \, \myvec{n}_e - \myvec{g} \big\Vert_{0,e}^2 
 + \big\Vert \operatorname{dev}\big( \mymatrix{\sigma}( \myvec{u}_N,\mymatrix{p}_N ) - \myspace{H} \mymatrix{p}_N\big) - \mymatrix{\lambda}_N \big\Vert_{0,\Omega}^2.
\end{align*}
Hence, the estimate \eqref{eq:residual_est_reliable} follows with the triangle inequality and Theorem~\ref{lem:a_posteriori_error}.
\end{proof}

An efficiency estimate for the error estimator $\eta^2(\mymatrix{\mu})$, where we unfortunately cannot avoid its suboptimality (as we make use of Lemma~\ref{lem:optimal_lambda} and Theorem~\ref{thm:efficiency}), can be derived as follows.

\begin{lemma}\label{lem:localDev}
There holds
\begin{align*}
 \big\Vert \operatorname{dev}\big(\mymatrix{\sigma} ( \myvec{u}_N,\mymatrix{p}_N ) - \myspace{H} \, \mymatrix{p}_N \big) - \mymatrix{\lambda}_N \big\Vert_{0,\Omega}^2 
 \lesssim \Vert ( \myvec{u}-\myvec{u}_N, \mymatrix{p}-\mymatrix{p}_N ) \Vert^2 + \Vert \mymatrix{\lambda} - \mymatrix{\lambda}_N \Vert_{0,\Omega}^2.
\end{align*}
\end{lemma}
\begin{proof}
Applying \eqref{eq:repres_continuousLambda}, the triangle inequality and the uniform boundedness of $\myspace{C}$ and $\myspace{H}$ yield
\begin{align*}
    \big\Vert \operatorname{dev}\big(\mymatrix{\sigma} ( \myvec{u}_N,\mymatrix{p}_N ) - \myspace{H} \, \mymatrix{p}_N \big) - \mymatrix{\lambda}_N \big\Vert_{0,\Omega}^2 
    & = \big\Vert \operatorname{dev}\big( \mymatrix{\sigma}(\myvec{u}_N,\mymatrix{p}_N) - \myspace{H} \, \mymatrix{p}_N \big) - \operatorname{dev}\big( \mymatrix{\sigma}(\myvec{u},\mymatrix{p}) - \myspace{H} \, \mymatrix{p} \big) + \mymatrix{\lambda} - \mymatrix{\lambda}_N \big\Vert_{0,\Omega}^2 \\
   & \lesssim \Vert \mymatrix{\sigma}(\myvec{u}-\myvec{u}_N,\mymatrix{p}-\mymatrix{p}_N)\Vert_{0,\Omega}^2 + \Vert  \myspace{H} \, (\mymatrix{p} -  \mymatrix{p}_N ) \Vert_{0,\Omega}^2 + \Vert \mymatrix{\lambda} - \mymatrix{\lambda}_N \Vert_{0,\Omega}^2\\
  & \lesssim \Vert ( \myvec{u}-\myvec{u}_N, \mymatrix{p}-\mymatrix{p}_N ) \Vert^2 + \Vert \mymatrix{\lambda} - \mymatrix{\lambda}_N \Vert_{0,\Omega}^2,
\end{align*}
which completes the argument.
\end{proof}

\begin{lemma}\label{lem:localEta}
Let the elasticity tensor $\myspace{C}$ be constant and let the transformations $\myvec{F}_T(\cdot)$ be affine for $T\in\mathcal{T}_h$. For any $\varepsilon > 0$ there holds
\begin{align}\label{eq:residual_est_effi3}
 \eta_T^2 
 \lesssim_{\varepsilon} p_T^{1+2\varepsilon} \left( p_T \, \Vert \myvec{u}^* - \myvec{u}_N\Vert_{1,\omega_T}^2 + p_T \, \Vert \mymatrix{p}^* - \mymatrix{p}_N\Vert_{0,\omega_T}^2 + p_T^{2\varepsilon} \sum_{T' \in \omega_T} \osc^2_{T'} \right) \qquad 
 \forall \, T \in \mathcal{T}_h,
\end{align}
where $\omega_T := \big\lbrace T' \in \mathcal{T}_h \; ; \; T \text{ and } T' \text{ share at least one edge} \big\rbrace$. Furthermore, \eqref{eq:residual_est_effi3} remains true if $(\myvec{u}^*,\mymatrix{p}^*)$ is exchanged for $(\myvec{u},\mymatrix{p})$.
\end{lemma}
\begin{proof}
Inserting $(\myvec{v},\mymatrix{0})\in V\times Q$ in \eqref{eq:additional_problem} we obtain
\begin{align*}
 \big( \mymatrix{\sigma}(\myvec{u}^*,\mymatrix{p}^*), \mymatrix{\varepsilon}(\myvec{v})\big)_{0,\Omega} = \ell(\myvec{v})
 \qquad \forall \, \myvec{v}\in V.
\end{align*}
Thereby, \eqref{eq:residual_est_effi3} follows from arguments used in the derivation of efficient residual-based a~posteriori error estimators, see, e.g.~\cite[Lem.~3.5]{melenk2001residual}. The additional claim follows by inserting $(\myvec{v},\mymatrix{0})\in V\times Q$ in \eqref{eq:mixed_variationalF_01} and using the same arguments as before for efficieny estimates.
\end{proof}

\begin{theorem}\label{thm:globalLowerEst}
Under the assumptions of Lemma~\ref{lem:localEta} there holds
\begin{align*}
 \eta^2(\mymatrix{\mu}^*)
 \lesssim_p \Vert (\myvec{u}-\myvec{u}_N, \mymatrix{p}-\mymatrix{p}_N)\Vert^2 + \Vert \mymatrix{\lambda}-\mymatrix{\lambda}_N\Vert_{0,\Omega}^2 + \Vert \mymatrix{p}-\mymatrix{p}_N\Vert_{0,\Omega} + \osc^2.
\end{align*}
\end{theorem}
\begin{proof}
The assertion immediately follows from Theorem~\ref{thm:efficiency}, Lemma~\ref{lem:localDev} and Lemma~\ref{lem:localEta}, where the expression in \eqref{eq:residual_est_effi3} is summed over $\mathcal{T}_h$.
\end{proof}

\begin{remark*}
The assertion of Thereom~\ref{thm:globalLowerEst} can also be proven by applying Lemma~\ref{lem:optimal_lambda} together with Lemma~\ref{lem:localDev} and Lemma~\ref{lem:localEta} by inserting $(\myvec{u},\mymatrix{p})$ in \eqref{eq:residual_est_effi3}, i.e.~the auxiliary problem \eqref{eq:additional_problem} is not actually needed to derive the above efficieny estimate. Furthermore, 
the assumption on $\myspace{C}$ in Lemma~\ref{lem:localEta} might be dropped at the cost of additional data oscillation terms.
\end{remark*}

Note that Lemma~\ref{lem:localEta} admits a local efficiency estimate: By using the same arguments as in Section~\ref{subsec:AuxProSec} with respect to $T\in \mathcal{T}_h$ we observe that 
\begin{align*}
 \E_T(\mymatrix{\mu}^*) 
 = \min_{\mymatrix{\mu}\in \Lambda_T} \E_T(\mymatrix{\mu})
\end{align*}
with $\Lambda_T := \left\lbrace \mymatrix{\mu}\in Q_{|T} \; ; \; \vert\mymatrix{\mu}\vert_F \leq \sigma_y \text{ a.e.~in } T \right\rbrace$. Thus, analogously to Lemma~\ref{lem:optimal_lambda} we obtain
\begin{align*}
 \E_T(\mymatrix{\mu}^*) 
 \leq \Vert \mymatrix{\lambda}-\mymatrix{\lambda}_N\Vert_{0,T}^2 + \big( \Vert \sigma_y \Vert_{0,T} + \Vert \mymatrix{\lambda}\Vert_{0,T} \big) \, \Vert \mymatrix{p}-\mymatrix{p}_N\Vert_{0,T}.
\end{align*}
By using the same estimation techniques as in the proof of Lemma~\ref{lem:localDev} we obtain
\begin{align*}
 \big\Vert \operatorname{dev}\big(\mymatrix{\sigma} ( \myvec{u}_N,\mymatrix{p}_N ) - \myspace{H} \, \mymatrix{p}_N \big) - \mymatrix{\lambda}_N \big\Vert_{0,T}^2 
 \lesssim \Vert \myvec{u}-\myvec{u}_N\big\Vert^2_{0,T} + \Vert \mymatrix{p}-\mymatrix{p}_N\big\Vert^2_{0,T} + \Vert \mymatrix{\lambda} - \mymatrix{\lambda}_N \Vert_{0,T}^2.
\end{align*}
This eventually gives the local efficiency estimate
\begin{align*}
 \eta^2_T(\mymatrix{\mu}^*)
 \lesssim_p \Vert \myvec{u}-\myvec{u}_N \Vert^2_{0,T} + \Vert \mymatrix{p}-\mymatrix{p}_N\Vert^2_{0,T}  + \Vert \mymatrix{\lambda}-\mymatrix{\lambda}_N\Vert_{0,T}^2 +  \Vert \mymatrix{p}-\mymatrix{p}_N\Vert_{0,T}+ \osc^2_T.
\end{align*}

% ----------------------------------------------------------------------------------------

\begin{figure}[ht]
  \centering 
  \begin{subfigure}[t]{0.3\textwidth}
    \centering
	\includegraphics[trim = 30mm 37mm 0mm 31mm, clip,height=40mm, keepaspectratio]{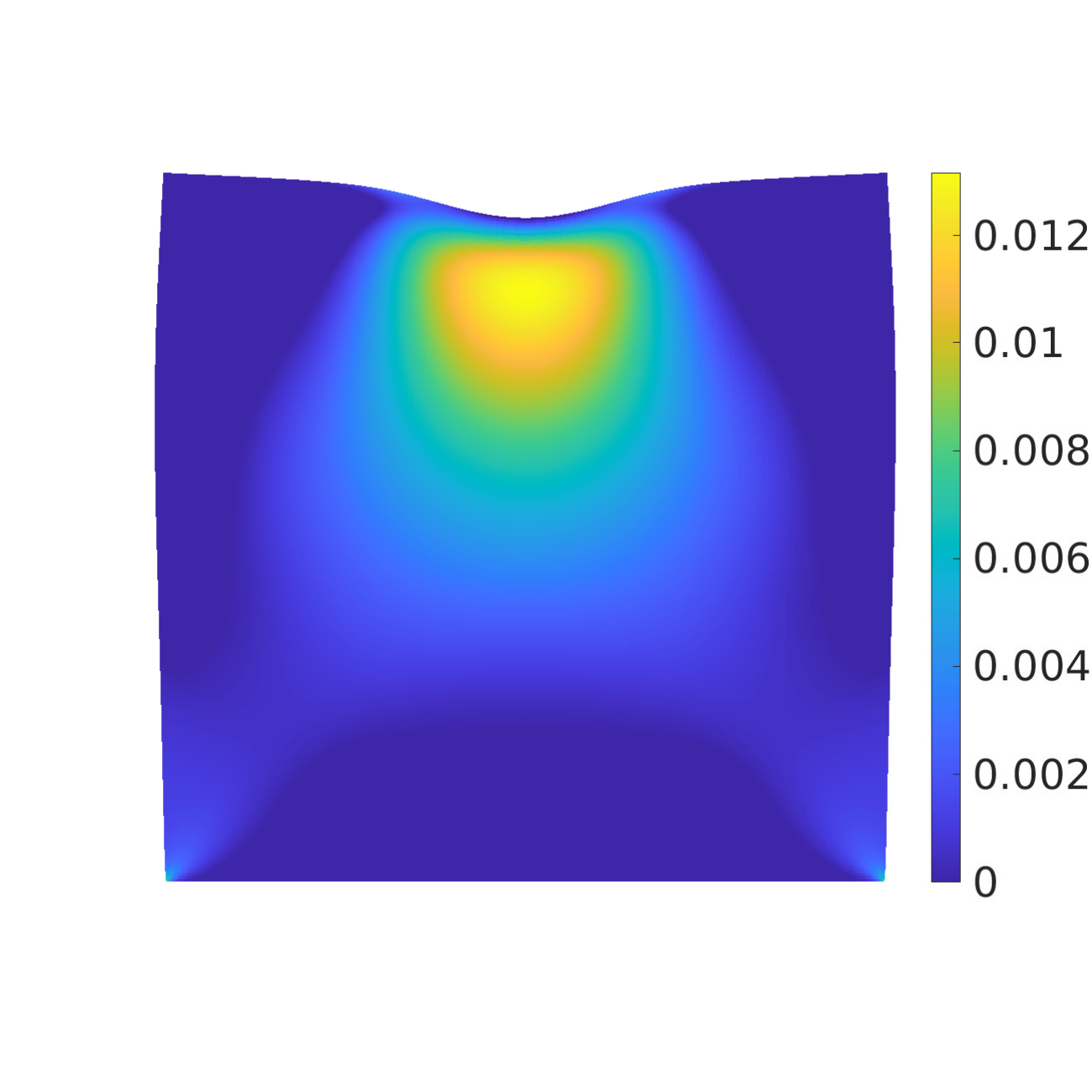}
	\caption{$\vert  \mymatrix{p}_{hp}  \vert_F$}
  \end{subfigure}
  %
  % -----------------------
  %
  \begin{subfigure}[t]{0.3\textwidth}
    \centering
	\includegraphics[trim = 30mm 37mm 15mm 31mm, clip,height=40mm, keepaspectratio]{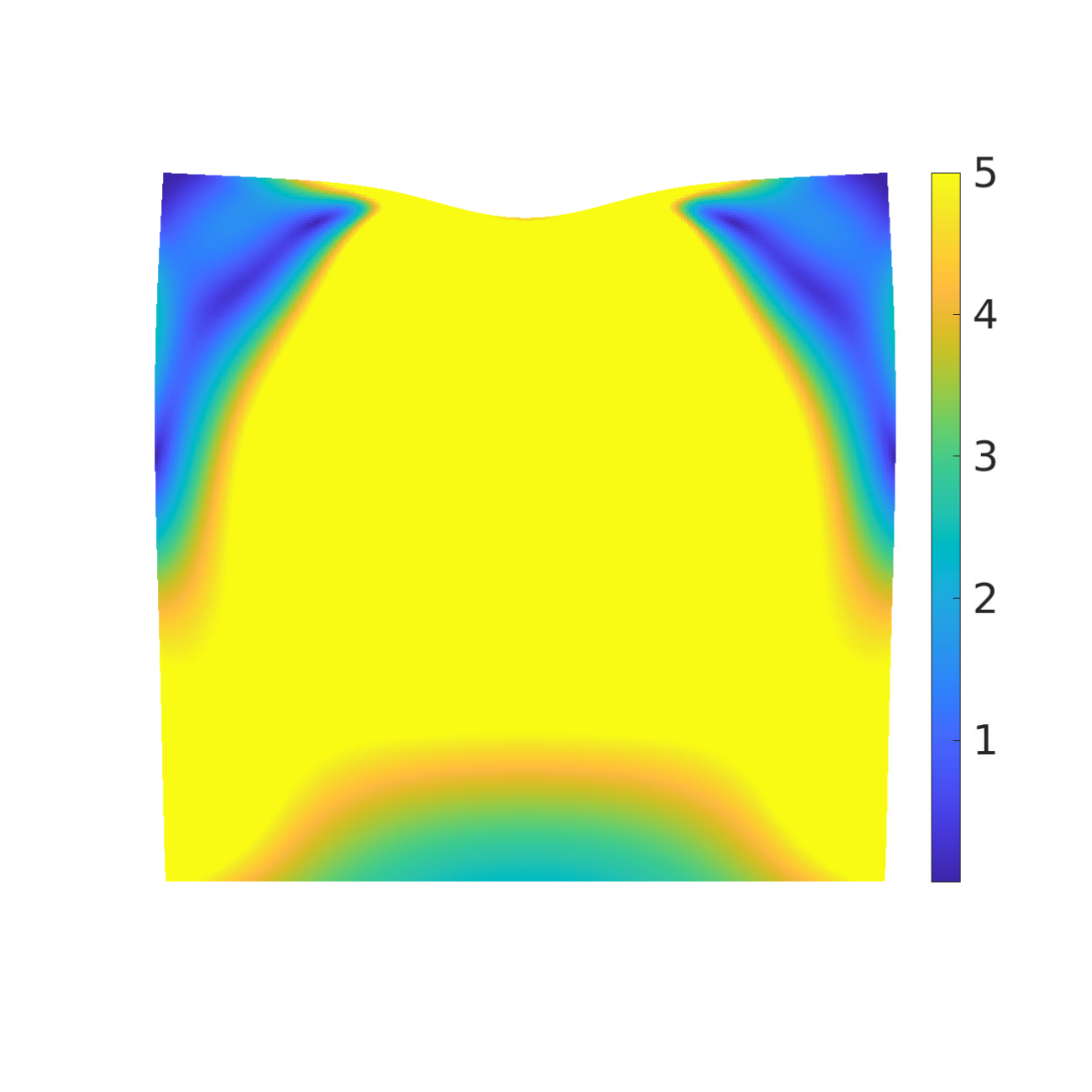}
	\caption{$\vert  \mymatrix{\lambda}_{hp} \vert_F$}
  \end{subfigure}
  %
  % -----------------------
  %
  \begin{subfigure}[t]{0.3\textwidth}
    \centering
	\includegraphics[trim = 45mm 45mm 37mm 38mm, clip,height=40mm, keepaspectratio]{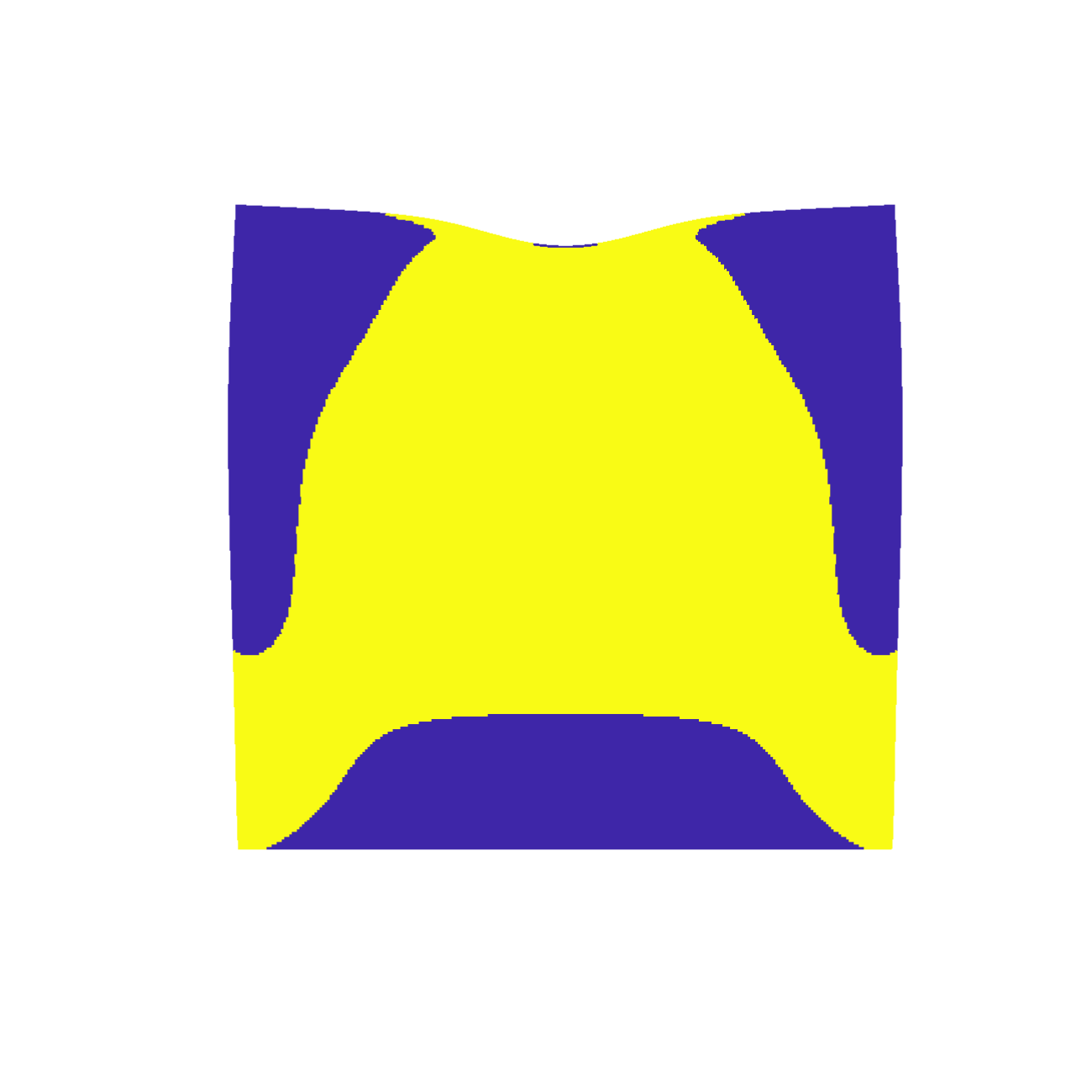}
	\caption{Purly elastic region (blue)}
	 \label{fig:elastic_region}
  \end{subfigure}
  \caption{\it Deformation of $\Omega$ magnified by factor 10 for uniform mesh with $h=2^{-7}$ and $p=1$.}
  \label{fig:solution}
\end{figure}

% ----------------------------------------------------------------------------------------

% \medskip

% ---------------------------------------------------------------------------------------
%   SECTION: Numerical results
% ----------------------------------------------------------------------------------------

\section{Numerical Results}
\label{sec:numeric}

In this section we consider some numerical examples in 2D, which illustrate the applicability of the discrete mixed formulation \eqref{eq:discrete_mixed_variationalF} (and also of the discrete variational inequality due to its equivalence to the mixed discretizations). In particular, $h$- and $hp$-adaptive schemes steered by the residual based a~posteriori error estimator \eqref{eq:eta_mu} are discussed. For this purpose, we set $\Omega := (-1,1)^2$ and $\Gamma_D := [-1,1]\times\lbrace -1 \rbrace$ and define $\myvec{f} := \myvec{0}$ on $\Omega$ and $\myvec{g} := (0,-400 \min(0,x_1^2-1/4)^2)^{\top}$ on $[-1,1]\times \lbrace 1\rbrace$ and zero elsewhere. Furthermore, we let $\myspace{C}\mymatrix{\tau} := \lambda \operatorname{tr}(\mymatrix{\tau})\mymatrix{I}+2\mu\mymatrix{\tau}$ with Lam\'e constants $\lambda:=\mu:=1000$, $\myspace{H}\mymatrix{\tau}:=500\mymatrix{\tau}$ and $\sigma_y:=5$. 
We refer to \cite{Bammer2023Apriori}, where this setting is already used for numerical experiments in the context of the discrete mixed formulation \eqref{eq:discrete_mixed_variationalF} with $\Lambda_{hp}^{(i)}=\Lambda_{hp}^{(s)}$. As no analytic solution is known, we use
 \begin{align*}
  e_{\myvec{u}}:=\| \myvec{u}_\text{fine} - \myvec{u}_{hp}\|_{1, \Omega}, \quad
  e_{\mymatrix{p}}:=\|\mymatrix{p}_\text{fine}-\mymatrix{p}_{hp} \|_{0, \Omega}, \quad
  e_{\mymatrix{\lambda}}:=\| \mymatrix{\lambda}_\text{fine}- \mymatrix{\lambda}_{hp} \|_{0, \Omega}
 \end{align*}
to quantify the approximation error. Thereby, \textit{fine} indicates an overkill discrete solution by halving $h_T$ and increasing $p_T$ by one on all elements of the finest mesh. In Figure~\ref{fig:solution}, a numerical solution computed for the above described setting is depicted. In particular, Figure~\ref{fig:elastic_region} shows the free boundary indicating the transition from pure elastic ($\mymatrix{p}=0$) to elastoplastic deformation ($\mymatrix{p}\neq 0$). In Figure~\ref{fig:error_2d}, the total error 
\begin{align}\label{eq:error}
 	\sqrt{e_{\myvec{u}}^2+e_{\mymatrix{p}}^2+e_{\mymatrix{\lambda}}^2}
 \end{align}
is plotted against the degrees of freedom (DOF). Figure~\ref{fig:errorEst_2d} shows the error estimator with respect to the DOF for uniform $h$-refinements with $p=1,2,3$, uniform $p$-refinements, $h$-adaptive refinements with $p=1,2,3$ and $hp$-adaptive refinements, where we use $\mymatrix{\lambda}_{hp}$ as $\mymatrix{\lambda}_N$, see Section \ref{sec:residualErrorEst}. For the $h$- and $hp$-adaptive refinements we apply Dörfler-marking with bulk parameter $0.5$ for the adaptive mesh refinements and a local regularity estimate based on the decay rate of the error estimator in $p$ for the $h$- vs.~$p$-decision, see \cite{banz2020posteriori,banz2018higher,Banz2021abstract}. We observe that all discretization schemes and all a posteriori error estimates converge at a certain algebraic rate (e.g.~1.5 with respect to DOF in the case of $hp$-adaptive refinements). We refer to \cite{Bammer2023Apriori} for more details on expectable convergence rates.

% ----------------------------------------------------------------------------------------

\begin{figure}%[ht]
  \centering 
  \begin{subfigure}[t]{0.3\textwidth}
    \centering
    \begin{tikzpicture}[scale=0.6]
		\begin{loglogaxis}[
			width=1.8\textwidth,
			mark size=2.5pt,
			line width=0.75pt,
			xmin=1e1,xmax=1e8,
			ymin=1e-6,ymax=1e+1,
			legend style={at={(0,0)},anchor=south west},
            legend style={fill=none}
			]
   
			\addplot+[mark=o, color=blue] table[x index=0,y expr=(sqrt(\thisrowno{3}^2+\thisrowno{4}^2+\thisrowno{5}^2))] {logos/h1.txt};
			\addplot+[mark=square, color=red] table[x index=0,y expr=(sqrt(\thisrowno{3}^2+\thisrowno{4}^2+\thisrowno{5}^2))] {logos/h2.txt};
            \addplot+[mark=x, color=green!60!black] table[x index=0,y expr=(sqrt(\thisrowno{3}^2+\thisrowno{4}^2+\thisrowno{5}^2))] {logos/h3.txt};
            \addplot+[mark=asterisk, color=darkbrown] table[x index=0,y expr=(sqrt(\thisrowno{3}^2+\thisrowno{4}^2+\thisrowno{5}^2))] {logos/p.txt};

			\addplot+[mark=+, color=purple] table[x index=0,y expr=(sqrt(\thisrowno{3}^2+\thisrowno{4}^2+\thisrowno{5}^2))] {logos/a1.txt};
            \addplot+[mark=o, solid, color=black] table[x index=0,y expr=(sqrt(\thisrowno{3}^2+\thisrowno{4}^2+\thisrowno{5}^2))] {logos/a2.txt};
            \addplot+[mark=10-pointed star, solid, color=teal] table[x index=0,y expr=(sqrt(\thisrowno{3}^2+\thisrowno{4}^2+\thisrowno{5}^2))] {logos/a3.txt};
            \addplot+[mark=diamond, solid, color=orange] table[x index=0,y expr=(sqrt(\thisrowno{3}^2+\thisrowno{4}^2+\thisrowno{5}^2))] {logos/hp.txt};

			\draw (1e6,9e-2) -- (1e7,9e-2);
			\draw (1e7,9e-2) -- (1e7,9e-2*0.316227766016838);
			\draw (1e6,9e-2) -- (1e7,9e-2*0.316227766016838);
			\node at (2.5e7,5.5e-2) {$0.5$};
			
			\draw (1e6,2e-3) -- (1e7,2e-3);
			\draw (1e7,2e-3) -- (1e7,2e-3*0.1);
			\draw (1e6,2e-3) -- (1e7,2e-3*0.1);
			\node at (1.8e7,8e-4) {$1$};
			
			\draw (5e5,2e-4*0.031622776601684) -- (5e6,2e-4*0.031622776601684);
 			\draw (5e5,2e-4) -- (5e5,2e-4*0.031622776601684);
 			\draw (5e5,2e-4) -- (5e6,2e-4*0.031622776601684);
			\node at (2e5,3.5e-5) {$1.5$};

   \legend{{$h1$},
   {$h2$},{$h3$},{$p$},
   {$a1$},{$a2$},{$a3$},{$hp$}}
      
		\end{loglogaxis}
	  \end{tikzpicture}
      \caption{$\sqrt{e_{\myvec{u}}^2+e_{\mymatrix{p}}^2+e_{\mymatrix{\lambda}}^2}$}
      \label{fig:error_2d}
    \end{subfigure}  
    \hspace{0.3cm} % -------------------------------------------------
    \begin{subfigure}[t]{0.3\textwidth}
      \centering
	  \begin{tikzpicture}[scale=0.6]
		\begin{loglogaxis}[
			width=1.8\textwidth,
			mark size=2.5pt,
			line width=0.75pt,
			xmin=1e1,xmax=1e8,
			ymin=1e-5,ymax=1e+2,
			legend style={at={(0,0)},anchor=south west},
            legend style={fill=none}
			]

   			\addplot+[mark=o, color=blue] table[x index=0,y index=6] {logos/h1.txt};
            \addplot+[mark=square, color=red] table[x index=0,y index=6] {logos/h2.txt};
            \addplot+[mark=x, color=green!60!black] table[x index=0,y index=6] {logos/h3.txt};
            \addplot+[mark=asterisk, color=darkbrown] table[x index=0,y index=6] {logos/p.txt};

			\addplot+[mark=+, color=purple] table[x index=0,y index=6] {logos/a1.txt};
            \addplot+[mark=o, solid, color=black] table[x index=0,y index=6] {logos/a2.txt};
            \addplot+[mark=10-pointed star, solid, color=teal] table[x index=0,y index=6] {logos/a3.txt};
            \addplot+[mark=diamond, solid, color=orange] table[x index=0,y index=6] {logos/hp.txt};
    
			\draw (1e6,9e-1) -- (1e7,9e-1);
			\draw (1e7,9e-1) -- (1e7,9e-1*0.346736850452532);
			\draw (1e6,9e-1) -- (1e7,9e-1*0.346736850452532);
			\node at (2.5e7,5.5e-1) {$0.46$};

			\draw (1e6,1e-1*0.464158883361278) -- (1e7,1e-1*0.464158883361278);
 			\draw (1e6,1e-1) -- (1e6,1e-1*0.464158883361278);
 			\draw (1e6,1e-1) -- (1e7,1e-1*0.464158883361278);
			\node at (6e5,7e-2) {$\frac{1}{3}$};
   
			\draw (1e6,1e-2) -- (1e7,1e-2);
			\draw (1e7,1e-2) -- (1e7,1e-2*0.1);
			\draw (1e6,1e-2) -- (1e7,1e-2*0.1);
			\node at (1.8e7,4e-3) {$1$};
			
			\draw (3e5,6e-4*0.031622776601684) -- (3e6,6e-4*0.031622776601684);
 			\draw (3e5,6e-4) -- (3e5,6e-4*0.031622776601684);
 			\draw (3e5,6e-4) -- (3e6,6e-4*0.031622776601684);
			\node at (1e5,9e-5) {$1.5$};

            \legend{{$h1$},{$h2$},{$h3$},{$p$},{$a1$},{$a2$},{$a3$},{$hp$}}
   
		\end{loglogaxis}
	   \end{tikzpicture}
       \caption{Error estimator}
       \label{fig:errorEst_2d}
     \end{subfigure}
     \hspace{0.3cm} % -------------------------------------------------
     \begin{subfigure}[t]{0.3\textwidth}
       \centering
	   \begin{tikzpicture}[scale=0.6] 
		\begin{loglogaxis}[
			width=1.8\textwidth,
			mark size=2.5pt,
			line width=0.75pt,
			xmin=1e1,xmax=1e8,
			ymin=8e-1,ymax=1e2,
			legend style={at={(0,1)},anchor=north west},
            legend style={fill=none}
			]

   			\addplot+[mark=o, color=blue] table[x index=0,y expr=(\thisrowno{6}/sqrt(\thisrowno{3}^2+\thisrowno{4}^2+\thisrowno{5}^2))] {logos/h1.txt};
            \addplot+[mark=square, color=red] table[x index=0,y expr=(\thisrowno{6}/sqrt(\thisrowno{3}^2+\thisrowno{4}^2+\thisrowno{5}^2))] {logos/h2.txt};
            \addplot+[mark=x, color=green!60!black] table[x index=0,y expr=(\thisrowno{6}/sqrt(\thisrowno{3}^2+\thisrowno{4}^2+\thisrowno{5}^2))] {logos/h3.txt};
            \addplot+[mark=asterisk, color=darkbrown] table[x index=0,y expr=(\thisrowno{6}/sqrt(\thisrowno{3}^2+\thisrowno{4}^2+\thisrowno{5}^2))] {logos/p.txt};

			\addplot+[mark=+, color=purple] table[x index=0,y expr=(\thisrowno{6}/sqrt(\thisrowno{3}^2+\thisrowno{4}^2+\thisrowno{5}^2))] {logos/a1.txt};
            \addplot+[mark=o, solid, color=black] table[x index=0,y expr=(\thisrowno{6}/sqrt(\thisrowno{3}^2+\thisrowno{4}^2+\thisrowno{5}^2))] {logos/a2.txt};
            \addplot+[mark=10-pointed star, solid, color=teal] table[x index=0,y expr=(\thisrowno{6}/sqrt(\thisrowno{3}^2+\thisrowno{4}^2+\thisrowno{5}^2))] {logos/a3.txt};
            \addplot+[mark=diamond, solid, color=orange] table[x index=0,y expr=(\thisrowno{6}/sqrt(\thisrowno{3}^2+\thisrowno{4}^2+\thisrowno{5}^2))] {logos/hp.txt};

			\draw (1e4,12*4.265795188015927) -- (1e5,12*4.265795188015927);
			\draw (1e4,12) -- (1e4,12*4.265795188015927);
			\draw (1e4,12) -- (1e5,12*4.265795188015927);
			\node at (3e3,25) {$0.63$};

			\draw (1e6,17*1.778279410038923) -- (1e7,17*1.778279410038923);
			\draw (1e6,17) -- (1e6,17*1.778279410038923);
			\draw (1e6,17) -- (1e7,17*1.778279410038923);
			\node at (5e5,24) {$\frac{1}{4}$};
   
            \legend{{$h1$},{$h2$},{$h3$},{$p$},{$a1$},{$a2$},{$a3$},{$hp$}}
      
		\end{loglogaxis}
	  \end{tikzpicture}
      \caption{Efficiency index}
      \label{fig:Effi_2d}
    \end{subfigure}

    \caption{\it Individual approximation errors vs.~degrees of freedom. In the legend $hi$ stands for uniform $h$-mesh refinements with $p=i$, $p$ for uniform $p$-refinements with $h=0.4$, $ai$ for $h$-adaptive refinements with $p=i$ and $hp$ for $hp$-adaptive refinements.}
  \label{fig:error_and_Est_2d}
    
\end{figure}
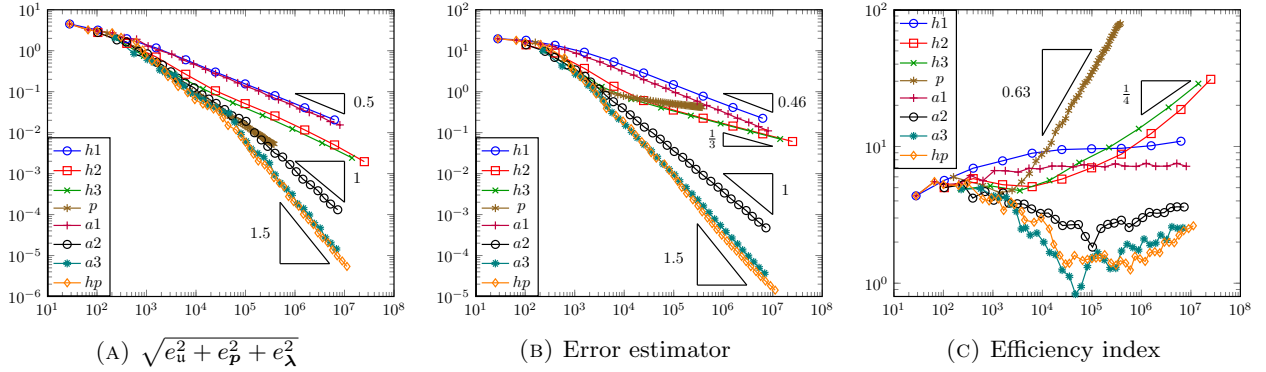

% ----------------------------------------------------------------------------------------

In Figure~\ref{fig:Effi_2d}, the efficiency indices are shown, i.e.~the quotient of the error estimator $\eta(\mymatrix{\mu}^*)$ as defined in \eqref{eq:eta_mu} and the total error \eqref{eq:error}. They are nearly constant for the $h$-refinements with $p=1$ and bounded for all three $h$-adaptive refinements and also for the $hp$-adaptive refinements. However, the efficiency indices increase for uniform $h$-refinements with $p=2,3$ and for the uniform $p$-refinements. A closer look at the individual contributions of the total error and the a posteriori error estimator shows the reason for the increase in the case of the uniform $h$-refinements: We observe in Figure~\ref{fig:Contribution_h2} and Figure~\ref{fig:Contribution_h3} that the contribution $\eta := \big(\sum_{T\in\mathcal{T}_h} \eta_T^2 \big)^{1/2}$  with $\eta_T$ defined in \eqref{eq:eta_T} converges with the rate $0.34$ and $0.33$, respectively, and dominates the error estimator. The contribution $e_{\mymatrix{\lambda}}$ has the rate $0.54$ and dominates the total error. Thus, the different rates of $\eta$ and $e_{\mymatrix{\lambda}}$ lead to the behavior of the efficiency indices. Note that $e_{\myvec{u}}$ and $e_{\mymatrix{p}}$ have the same rate as $\eta$. As the rate of $e_{\mymatrix{\lambda}}$ is significanly larger the efficiency indices may (asymptotically) be bounded as well. The increasing efficiency indices in the case of uniform $p$-refinements may reflect the $p$-dependency of the efficiency estimates of Lemma~\ref{lem:localEta} and Theorem~\ref{thm:globalLowerEst}, see Figure~\ref{fig:Contribution_p}. We clearly see in the Figures~\ref{fig:Contribution_h1}, \ref{fig:Contribution_a1}, \ref{fig:Contribution_a2}, \ref{fig:Contribution_a3} and \ref{fig:Contribution_hp} that 
the convergence rates of the error contributions for all other refinements are nearly the same, which explains the constant/bounded behavior of the corresponding efficiency indices.

% ----------------------------------------------------------------------------------------

\begin{figure}%[tb]
  \centering
  \begin{subfigure}[t]{0.25\textwidth}
    \centering
	\begin{tikzpicture}[scale=0.5]  
		\begin{loglogaxis}[
			width=2\textwidth,
			mark size=2.5pt,
			line width=0.75pt,
			xmin=1e1,xmax=1e7,
			ymin=1e-5,ymax=5e1,
			legend style={at={(0,0)},anchor=south west}
			]
   
			\addplot+[mark=x, solid, color=darkbrown] table[x index=0,y expr=(sqrt(\thisrowno{3}^2+\thisrowno{4}^2))] {logos/h1.txt};
            \addplot+[mark=square, color=red] table[x index=0,y index=5] {logos/h1.txt};
            \addplot+[mark=o, color=orange] table[x index=0,y expr=(sqrt(\thisrowno{7}^2+\thisrowno{8}^2))] {logos/h1.txt};
            \addplot+[mark=+, color=teal] table[x index=0,y index=9] {logos/h1.txt};
            \addplot+[mark=diamond, solid, color=blue] table[x index=0,y index=10] {logos/h1.txt};

			\draw (2e5,5e-4) -- (2e6,5e-4);
			\draw (2e6,5e-4) -- (2e6,5e-4*0.346736850452532);
			\draw (2e5,5e-4) -- (2e6,5e-4*0.346736850452532);
			\node at (4.5e6,3e-4) {$0.46$};

   			\draw (2e5,2e-0) -- (2e6,2e-0);
			\draw (2e6,2e-0) -- (2e6,2e-0*0.346736850452532);
			\draw (2e5,2e-0) -- (2e6,2e-0*0.346736850452532);
			\node at (4.5e6,1.2e-0) {$0.46$};
   
            \legend{{$\sqrt{e_{\myvec{u}}^2+e_{\mymatrix{p}}^2}$},{$e_{\mymatrix{\lambda}}$},{$\eta$},{dev},{$\E(\mymatrix{\mu}^*)$}}
      
		\end{loglogaxis}
	  \end{tikzpicture}
      \caption{$h$-uniform, $p=1$}
      \label{fig:Contribution_h1}
    \end{subfigure}
    \hspace{-0.3cm} % -------------------------------------------------
    \begin{subfigure}[t]{0.25\textwidth}
      \centering
	  \begin{tikzpicture}[scale=0.5]
		\begin{loglogaxis}[
			width=2\textwidth,
			mark size=2.5pt,
			line width=0.75pt,
			xmin=5e1,xmax=5e7,
			ymin=1e-5,ymax=5e1,
			legend style={at={(0,0)},anchor=south west}
			]
	
			\addplot+[mark=x, solid, color=darkbrown] table[x index=0,y expr=(sqrt(\thisrowno{3}^2+\thisrowno{4}^2))] {logos/h2.txt};
            \addplot+[mark=square, color=red] table[x index=0,y index=5] {logos/h2.txt};
            \addplot+[mark=o, color=orange] table[x index=0,y expr=(sqrt(\thisrowno{7}^2+\thisrowno{8}^2))] {logos/h2.txt};
            \addplot+[mark=+, color=teal] table[x index=0,y index=9] {logos/h2.txt};
            \addplot+[mark=diamond, solid, color=blue] table[x index=0,y index=10] {logos/h2.txt};

			\draw (5e5,2e-4) -- (5e6,2e-4);			\draw (5e6,2e-4) -- (5e6,2e-4*0.457088189614875);
			\draw (5e5,2e-4) -- (5e6,2e-4*0.457088189614875);
			\node at (1.3e7,1.4e-4) {$0.34$};

			\draw (5e5,4e-1) -- (5e6,4e-1);
			\draw (5e6,4e-1) -- (5e6,4e-1*0.457088189614875);
			\draw (5e5,4e-1) -- (5e6,4e-1*0.457088189614875);
			\node at (1.3e7,2.9e-1) {$0.34$};

			\draw (5e5,3e-3*0.288403150312661) -- (5e6,3e-3*0.288403150312661);
 			\draw (5e5,3e-3) -- (5e5,3e-3*0.288403150312661);
 			\draw (5e5,3e-3) -- (5e6,3e-3*0.288403150312661);
			\node at (2e5,1.5e-3) {$0.54$};   
   
            \legend{{$\sqrt{e_{\myvec{u}}^2+e_{\mymatrix{p}}^2}$},{$e_{\mymatrix{\lambda}}$},{$\eta$},{dev},{$\E(\mymatrix{\mu}^*)$}}
   
		\end{loglogaxis}
	\end{tikzpicture}
      \caption{$h$-uniform, $p=2$}
      \label{fig:Contribution_h2}
    \end{subfigure}
    \hspace{-0.3cm} % -------------------------------------------------
    \begin{subfigure}[t]{0.25\textwidth}
    \centering
	\begin{tikzpicture}[scale=0.5]
		\begin{loglogaxis}[
			width=2\textwidth,
			mark size=2.5pt,
			line width=0.75pt,
			xmin=1e2,xmax=5e7,
			ymin=1e-5,ymax=5e1,
			legend style={at={(0,0)},anchor=south west}
			]

			\addplot+[mark=x, solid, color=darkbrown] table[x index=0,y expr=(sqrt(\thisrowno{3}^2+\thisrowno{4}^2))] {logos/h3.txt};
            \addplot+[mark=square, color=red] table[x index=0,y index=5] {logos/h3.txt};
            \addplot+[mark=o, color=orange] table[x index=0,y expr=(sqrt(\thisrowno{7}^2+\thisrowno{8}^2))] {logos/h3.txt};
            \addplot+[mark=+, color=teal] table[x index=0,y index=9] {logos/h3.txt};
            \addplot+[mark=diamond, solid, color=blue] table[x index=0,y index=10] {logos/h3.txt};

			\draw (5e5,2e-4) -- (5e6,2e-4);
			\draw (5e6,2e-4) -- (5e6,2e-4*0.464162445926035);
			\draw (5e5,2e-4) -- (5e6,2e-4*0.464162445926035);
			\node at (1.3e7,1.4e-4) {$0.33$};

			\draw (5e5,4e-1) -- (5e6,4e-1);
			\draw (5e6,4e-1) -- (5e6,4e-1*0.464162445926035);
			\draw (5e5,4e-1) -- (5e6,4e-1*0.464162445926035);
			\node at (1.3e7,2.9e-1) {$0.33$};

			\draw (5e5,3e-3*0.288403150312661) -- (5e6,3e-3*0.288403150312661);
 			\draw (5e5,3e-3) -- (5e5,3e-3*0.288403150312661);
 			\draw (5e5,3e-3) -- (5e6,3e-3*0.288403150312661);
			\node at (2e5,1.5e-3) {$0.54$};  
   
            \legend{{$\sqrt{e_{\myvec{u}}^2+e_{\mymatrix{p}}^2}$},{$e_{\mymatrix{\lambda}}$},{$\eta$},{dev},{$\E(\mymatrix{\mu}^*)$}}
			
		\end{loglogaxis}
	  \end{tikzpicture}
      \caption{$h$-uniform, $p=3$}
      \label{fig:Contribution_h3}
    \end{subfigure}
    \hspace{-0.3cm} % -------------------------------------------------
    \begin{subfigure}[t]{0.25\textwidth}
    \centering
	\begin{tikzpicture}[scale=0.5]  
		\begin{loglogaxis}[
			width=2\textwidth,
			mark size=2.5pt,
			line width=0.75pt,
			xmin=1e2,xmax=1e6,
			ymin=1e-5,ymax=5e1,
			legend style={at={(0,0)},anchor=south west}
			]
   
			\addplot+[mark=x, solid, color=darkbrown] table[x index=0,y expr=(sqrt(\thisrowno{3}^2+\thisrowno{4}^2))] {logos/p.txt};
            \addplot+[mark=square, color=red] table[x index=0,y index=5] {logos/p.txt};
            \addplot+[mark=o, color=orange] table[x index=0,y expr=(sqrt(\thisrowno{7}^2+\thisrowno{8}^2))] {logos/p.txt};
            \addplot+[mark=+, color=teal] table[x index=0,y index=9] {logos/p.txt};
            \addplot+[mark=diamond, solid, color=blue] table[x index=0,y index=10] {logos/p.txt};

            \legend{{$\sqrt{e_{\myvec{u}}^2+e_{\mymatrix{p}}^2}$},{$e_{\mymatrix{\lambda}}$},{$\eta$},{dev},{$\E(\mymatrix{\mu}^*)$}}
      
		\end{loglogaxis}
	  \end{tikzpicture}
      \caption{$p$-uniform, $h=0.4$}
      \label{fig:Contribution_p}
    \end{subfigure}
    
\vspace{0.25cm} % -------------------------------------------------
    
    \begin{subfigure}[t]{0.25\textwidth}
      \centering
	  \begin{tikzpicture}[scale=0.5]  
		\begin{loglogaxis}[
			width=2\textwidth,
			mark size=2.5pt,
			line width=0.75pt,
			xmin=1e1,xmax=1e7,
			ymin=1e-5,ymax=5e1,
			legend style={at={(0,0)},anchor=south west}
			]
   
			\addplot+[mark=x, solid, color=darkbrown] table[x index=0,y expr=(sqrt(\thisrowno{3}^2+\thisrowno{4}^2))] {logos/a1.txt};
            \addplot+[mark=square, color=red] table[x index=0,y index=5] {logos/a1.txt};
            \addplot+[mark=o, color=orange] table[x index=0,y expr=(sqrt(\thisrowno{7}^2+\thisrowno{8}^2))] {logos/a1.txt};
            \addplot+[mark=+, color=teal] table[x index=0,y index=9] {logos/a1.txt};
            \addplot+[mark=diamond, solid, color=blue] table[x index=0,y index=10] {logos/a1.txt};

			\draw (2e5,4e-4) -- (2e6,4e-4);
			\draw (2e6,4e-4) -- (2e6,4e-4*0.316227766016838);
			\draw (2e5,4e-4) -- (2e6,4e-4*0.316227766016838);
			\node at (3.8e6,2.5e-4) {$\frac{1}{2}$};

			\draw (2e5,2e-0) -- (2e6,2e-0);
			\draw (2e6,2e-0) -- (2e6,2e-0*0.316227766016838);
			\draw (2e5,2e-0) -- (2e6,2e-0*0.316227766016838);
			\node at (3.8e6,1.2e-0) {$\frac{1}{2}$};
   
            \legend{{$\sqrt{e_{\myvec{u}}^2+e_{\mymatrix{p}}^2}$},{$e_{\mymatrix{\lambda}}$},{$\eta$},{dev},{$\E(\mymatrix{\mu}^*)$}}
      
		\end{loglogaxis}
	  \end{tikzpicture}
      \caption{$h$-adaptive, $p=1$}
      \label{fig:Contribution_a1}
    \end{subfigure}
    \hspace{-0.3cm} % -------------------------------------------------
    \begin{subfigure}[t]{0.25\textwidth}
      \centering
	  \begin{tikzpicture}[scale=0.5]  
		\begin{loglogaxis}[
			width=2\textwidth,
			mark size=2.5pt,
			line width=0.75pt,
			xmin=5e1,xmax=1e7,
			ymin=1e-7,ymax=5e1,
			legend style={at={(0,0)},anchor=south west}
			]
	
			\addplot+[mark=x, solid, color=darkbrown] table[x index=0,y expr=(sqrt(\thisrowno{3}^2+\thisrowno{4}^2))] {logos/a2.txt};
            \addplot+[mark=square, color=red] table[x index=0,y index=5] {logos/a2.txt};
            \addplot+[mark=o, color=orange] table[x index=0,y expr=(sqrt(\thisrowno{7}^2+\thisrowno{8}^2))] {logos/a2.txt};
            \addplot+[mark=+, color=teal] table[x index=0,y index=9] {logos/a2.txt};
            \addplot+[mark=diamond, solid, color=blue] table[x index=0,y index=10] {logos/a2.txt};

			\draw (2e5,2e-5) -- (2e6,2e-5);
			\draw (2e6,2e-5) -- (2e6,2e-5*0.1);
			\draw (2e5,2e-5) -- (2e6,2e-5*0.1);
			\node at (3e6,6e-6) {1};

			\draw (2e5,7e-2) -- (2e6,7e-2);
			\draw (2e6,7e-2) -- (2e6,7e-2*0.1);
			\draw (2e5,7e-2) -- (2e6,7e-2*0.1);
			\node at (3e6,2.5e-2) {1};
   
            \legend{{$\sqrt{e_{\myvec{u}}^2+e_{\mymatrix{p}}^2}$},{$e_{\mymatrix{\lambda}}$},{$\eta$},{dev},{$\E(\mymatrix{\mu}^*)$}}
   
		\end{loglogaxis}
	  \end{tikzpicture}
      \caption{$h$-adaptive, $p=2$}
      \label{fig:Contribution_a2}
    \end{subfigure}
    \hspace{-0.3cm} % -------------------------------------------------
    \begin{subfigure}[t]{0.25\textwidth}
      \centering
	  \begin{tikzpicture}[scale=0.5] 
		\begin{loglogaxis}[
			width=2\textwidth,
			mark size=2.5pt,
			line width=0.75pt,
			xmin=1e2,xmax=1e7,
			ymin=1e-8,ymax=5e1,
			legend style={at={(0,0)},anchor=south west}
			]

			\addplot+[mark=x, solid, color=darkbrown] table[x index=0,y expr=(sqrt(\thisrowno{3}^2+\thisrowno{4}^2))] {logos/a3.txt};
            \addplot+[mark=square, color=red] table[x index=0,y index=5] {logos/a3.txt};
            \addplot+[mark=o, color=orange] table[x index=0,y expr=(sqrt(\thisrowno{7}^2+\thisrowno{8}^2))] {logos/a3.txt};
            \addplot+[mark=+, color=teal] table[x index=0,y index=9] {logos/a3.txt};
            \addplot+[mark=diamond, solid, color=blue] table[x index=0,y index=10] {logos/a3.txt};

			\draw (2e5,5e-6) -- (2e6,5e-6);
			\draw (2e6,5e-6) -- (2e6,5e-6*0.039810717055350);
			\draw (2e5,5e-6) -- (2e6,5e-6*0.039810717055350);
			\node at (3.8e6,1e-6) {1.4};

			\draw (2e5,1e-2) -- (2e6,1e-2);
			\draw (2e6,1e-2) -- (2e6,1e-2*0.044668359215096);
			\draw (2e5,1e-2) -- (2e6,1e-2*0.044668359215096);
			\node at (4e6,3e-3) {1.35};
   
            \legend{{$\sqrt{e_{\myvec{u}}^2+e_{\mymatrix{p}}^2}$},{$e_{\mymatrix{\lambda}}$},{$\eta$},{dev},{$\E(\mymatrix{\mu}^*)$}}
			
		\end{loglogaxis}
	  \end{tikzpicture}
      \caption{$h$-adaptive, $p=3$}
      \label{fig:Contribution_a3}
    \end{subfigure}
    \hspace{-0.3cm} % -------------------------------------------------
    \begin{subfigure}[t]{0.25\textwidth}
      \centering
	\begin{tikzpicture}[scale=0.5]  
		\begin{loglogaxis}[
			width=2\textwidth,
			mark size=2.5pt,
			line width=0.75pt,
			xmin=1e1,xmax=5e7,
			ymin=1e-8,ymax=5e1,
			legend style={at={(0,0)},anchor=south west}
			]
	
			\addplot+[mark=x, solid, color=darkbrown] table[x index=0,y expr=(sqrt(\thisrowno{3}^2+\thisrowno{4}^2))] {logos/hp.txt};
            \addplot+[mark=square, color=red] table[x index=0,y index=5] {logos/hp.txt};
            \addplot+[mark=o, color=orange] table[x index=0,y expr=(sqrt(\thisrowno{7}^2+\thisrowno{8}^2))] {logos/hp.txt};
            \addplot+[mark=+, color=teal] table[x index=0,y index=9] {logos/hp.txt};
            \addplot+[mark=diamond, solid, color=blue] table[x index=0,y index=10] {logos/hp.txt};

			\draw (1e5,2e-6*0.031622776601684) -- (1e6,2e-6*0.031622776601684);
 			\draw (1e5,2e-6) -- (1e5,2e-6*0.031622776601684);
 			\draw (1e5,2e-6) -- (1e6,2e-6*0.031622776601684);
			\node at (4e4,4e-7) {$1.5$};

   			\draw (3e5,6e-4*0.031622776601684) -- (3e6,6e-4*0.031622776601684);
 			\draw (3e5,6e-4) -- (3e5,6e-4*0.031622776601684);
 			\draw (3e5,6e-4) -- (3e6,6e-4*0.031622776601684);
			\node at (1e5,9e-5) {$1.5$};

   % 			\draw (2e5,1e-2) -- (2e6,1e-2);
			% \draw (2e6,1e-2) -- (2e6,1e-2*0.044668359215096);
			% \draw (2e5,1e-2) -- (2e6,1e-2*0.044668359215096);
			% \node at (4e6,3e-3) {1.35};
   
            \legend{{$\sqrt{e_{\myvec{u}}^2+e_{\mymatrix{p}}^2}$},{$e_{\mymatrix{\lambda}}$},{$\eta$},{dev},{$\E^2(\mymatrix{\mu}^*)$}}
   
		\end{loglogaxis}
	 \end{tikzpicture}
     \caption{$hp$-adaptive}
     \label{fig:Contribution_hp}
   \end{subfigure}
   
   \caption{\it Individual approximation errors vs.~degrees of freedom. In the legend $hi$ stands for uniform $h$-refinements with $p=i$, $p$ for uniform $p$-refinements with $h=0.4$, $ai$ for $h$-adaptive refinements with $p=i$ and $hp$ for $hp$-adaptive refinements.}
  \label{fig:errorContr_2d}
  
\end{figure}
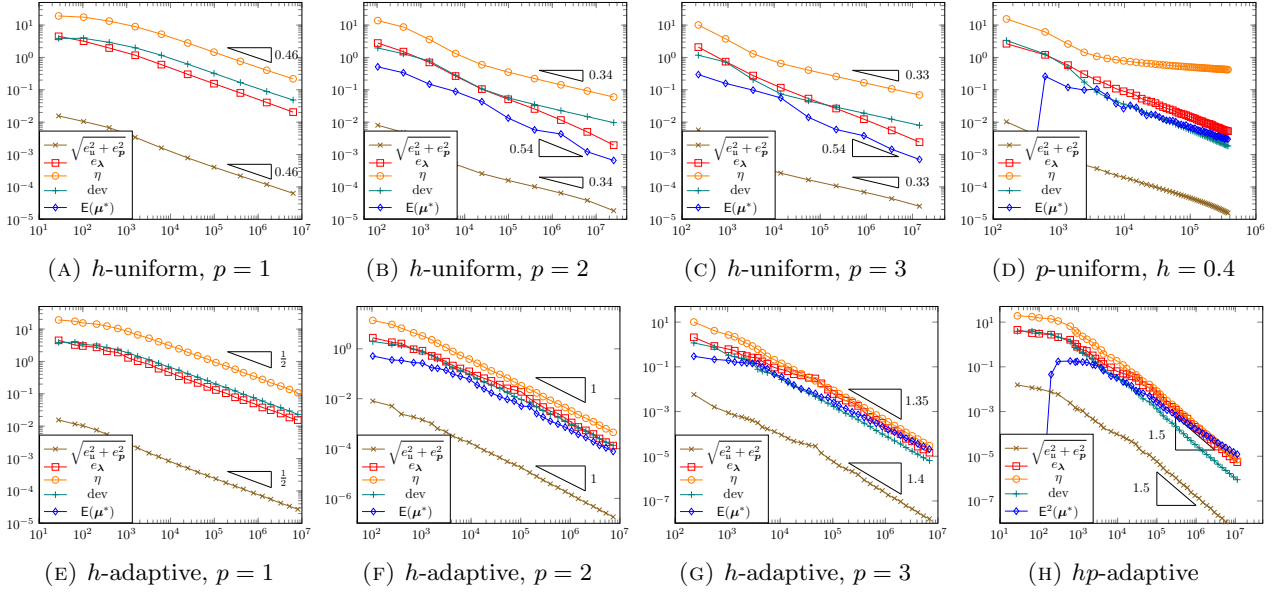

% ----------------------------------------------------------------------------------------

The convergence rates of the schemes with adaptive refinements are significantly superior to those with uniform refinements, as the associated finite element spaces are adapted to the singular behavior of the solution, see Figure~\ref{fig:meshes_2d}. They are even optimal for the uniform polynomial degree $p=1$ and $p=2$. In the case of $h$-adaptivity with $p=1$ only the Dirichlet-to-Neumann singularities are resolved by $h$-refinements indicating that the solution is sufficiently regular at the free boundary (which is not resolved). When $h$-adaptive refinements with $p=2,3$ are applied, the free boundary is also resolved by $h$-refinements, where the refinements become more local for $p=3$. By applying $hp$-adaptive refinements we get the typical $hp$-refinement pattern towards the corner singularities (including the Dirichlet-to-Neumann singularities) and towards the free boundary with a local polynomial degree of $p_T=4$ for those elements which are intersected by the free boundary. We emphasize that the $h$- vs.~$p$-decision strategy used to steer the $hp$-adaptive refinements is prone to $p$-refinements and leads to this relatively high local polynomial degree. Note that isotropic refinements of mesh elements with $p_T \geq 4$ are actually incapable to adequately resolve of the edge like singularity of the curved free boundary. Hence, the convergence rates are only algebraic of order $3/2$ (and not exponential), which is, however, the best rate in the numerical experiments.

% ----------------------------------------------------------------------------------------

\begin{figure}%[hb]
  \centering 
  \begin{subfigure}[t]{0.23\textwidth}
    \centering
    \includegraphics[trim = 115mm 107mm 95mm 95mm, clip,height=33.9mm, keepaspectratio]{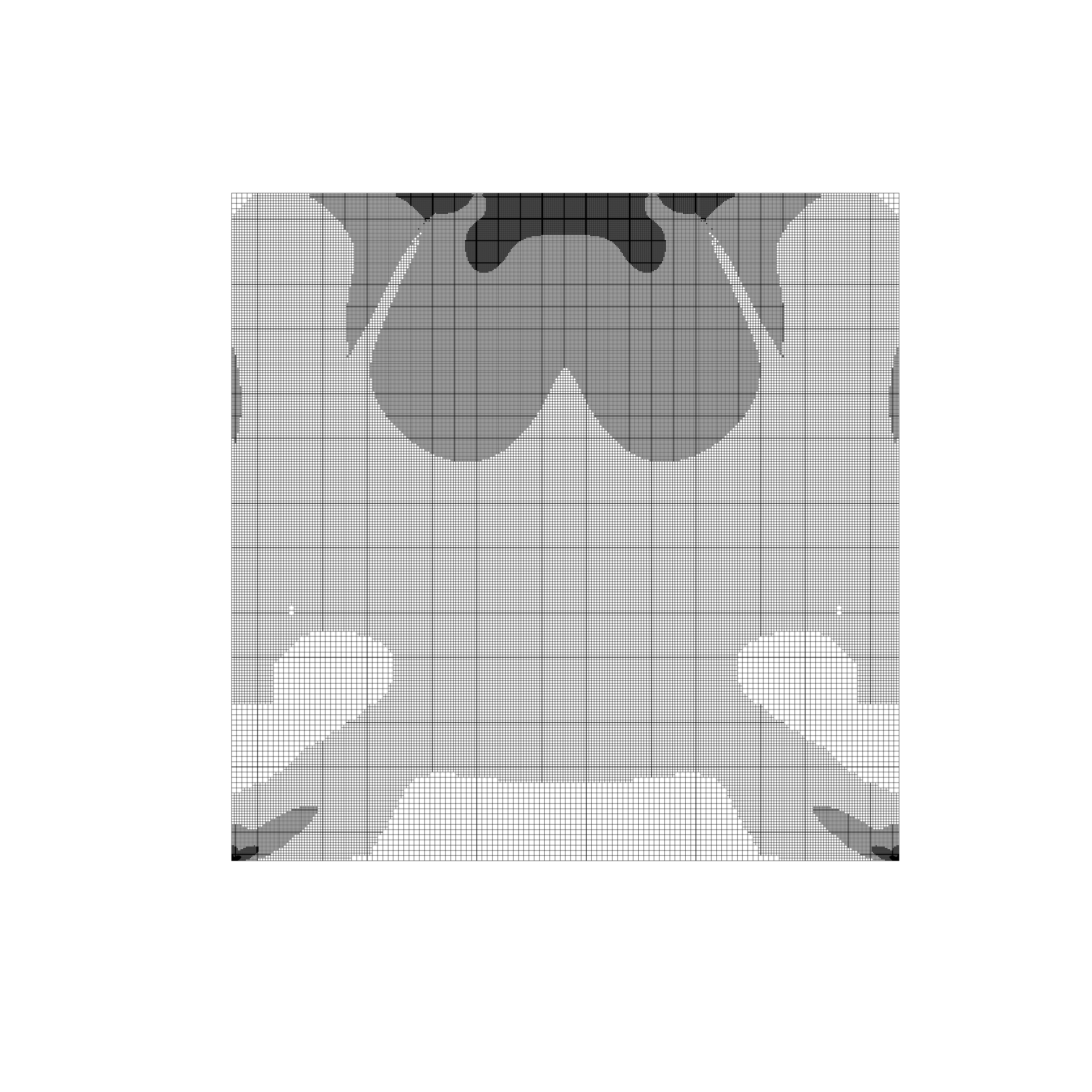}
    \caption{$h$-adaptive mesh, $p=1$ \\ (nr.~20, $783.512$ DOF)}
  \end{subfigure}
  \hspace{0.05cm} % -------------------------------------------------
  \begin{subfigure}[t]{0.23\textwidth}
    \centering
    \includegraphics[trim = 120mm 113mm 100mm 100mm, clip,height=34mm, keepaspectratio]{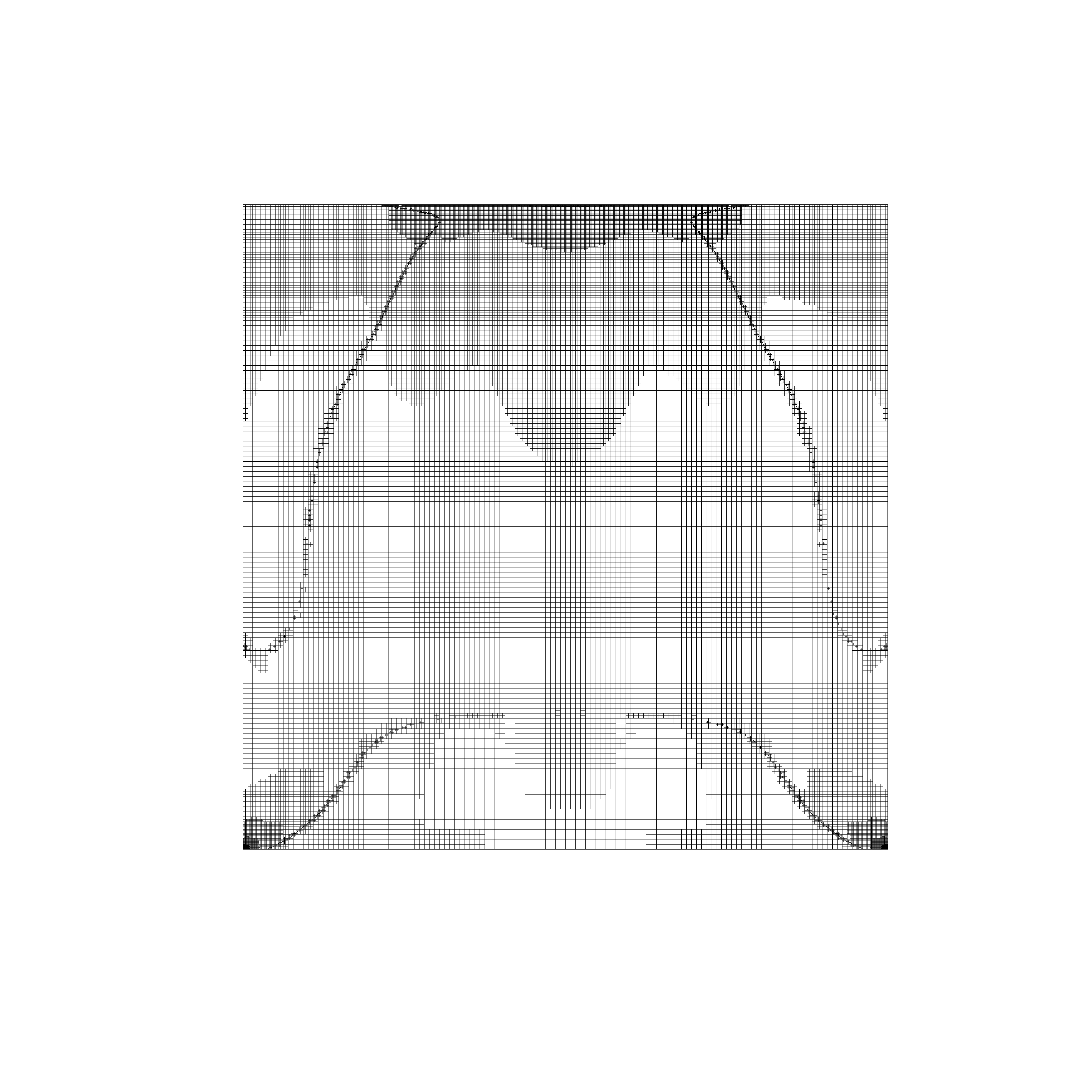}
    \caption{$h$-adaptive mesh, $p=2$ \\ (nr.~25, $1.060.608$ DOF)}
  \end{subfigure}
  \hspace{0.05cm} % -------------------------------------------------
  \begin{subfigure}[t]{0.23\textwidth}
    \centering
    \includegraphics[trim = 45mm 42mm 37mm 37mm, clip,height=34mm, keepaspectratio]{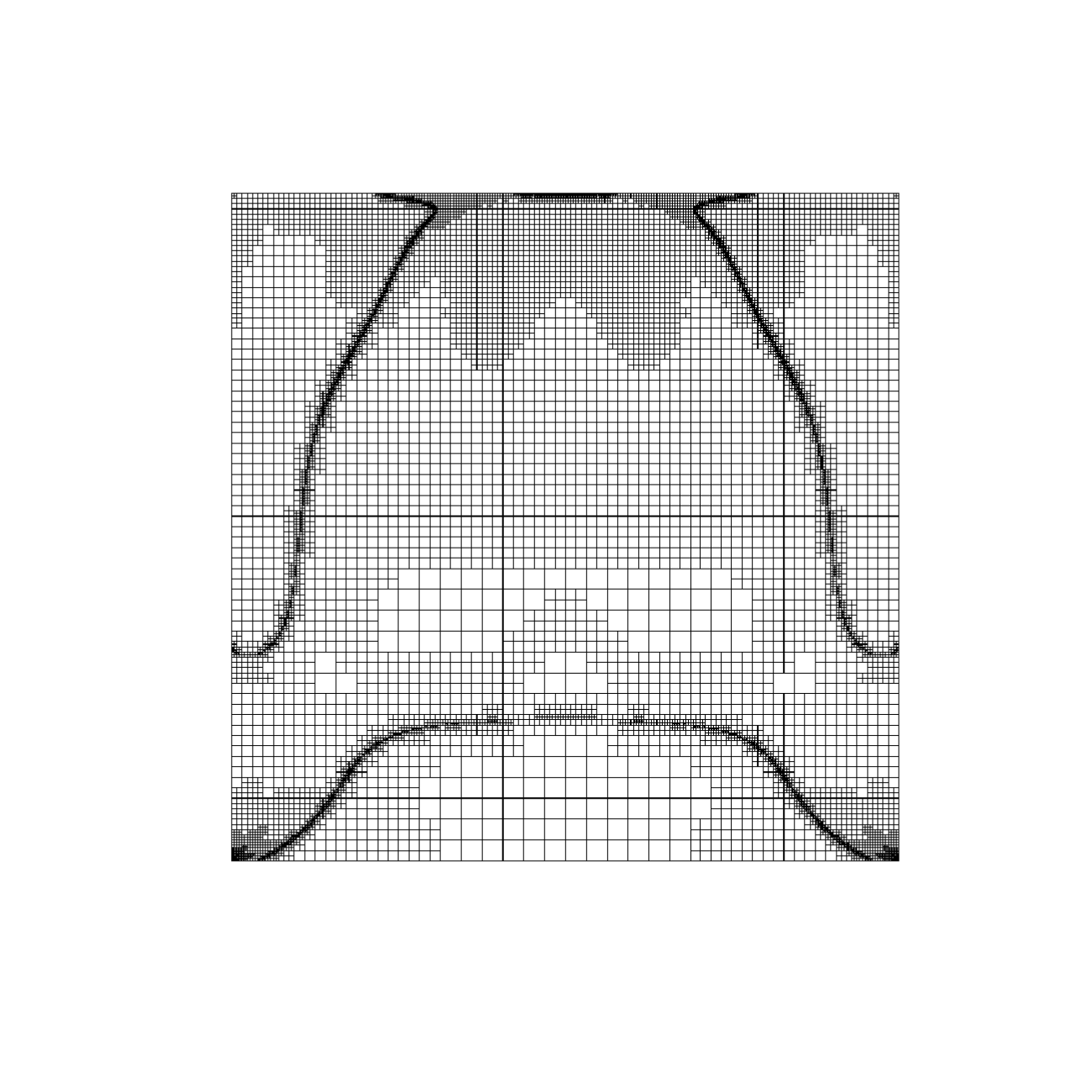}
    \caption{$h$-adaptive mesh, $p=3$ \\ (nr.~28, $874.164$ DOF)}
  \end{subfigure}
  \hspace{0.3cm} % -------------------------------------------------
  \begin{subfigure}[t]{0.2\textwidth}
    \centering
    \includegraphics[trim = 37mm 42mm 15mm 32mm, clip,height=35.2mm, keepaspectratio]{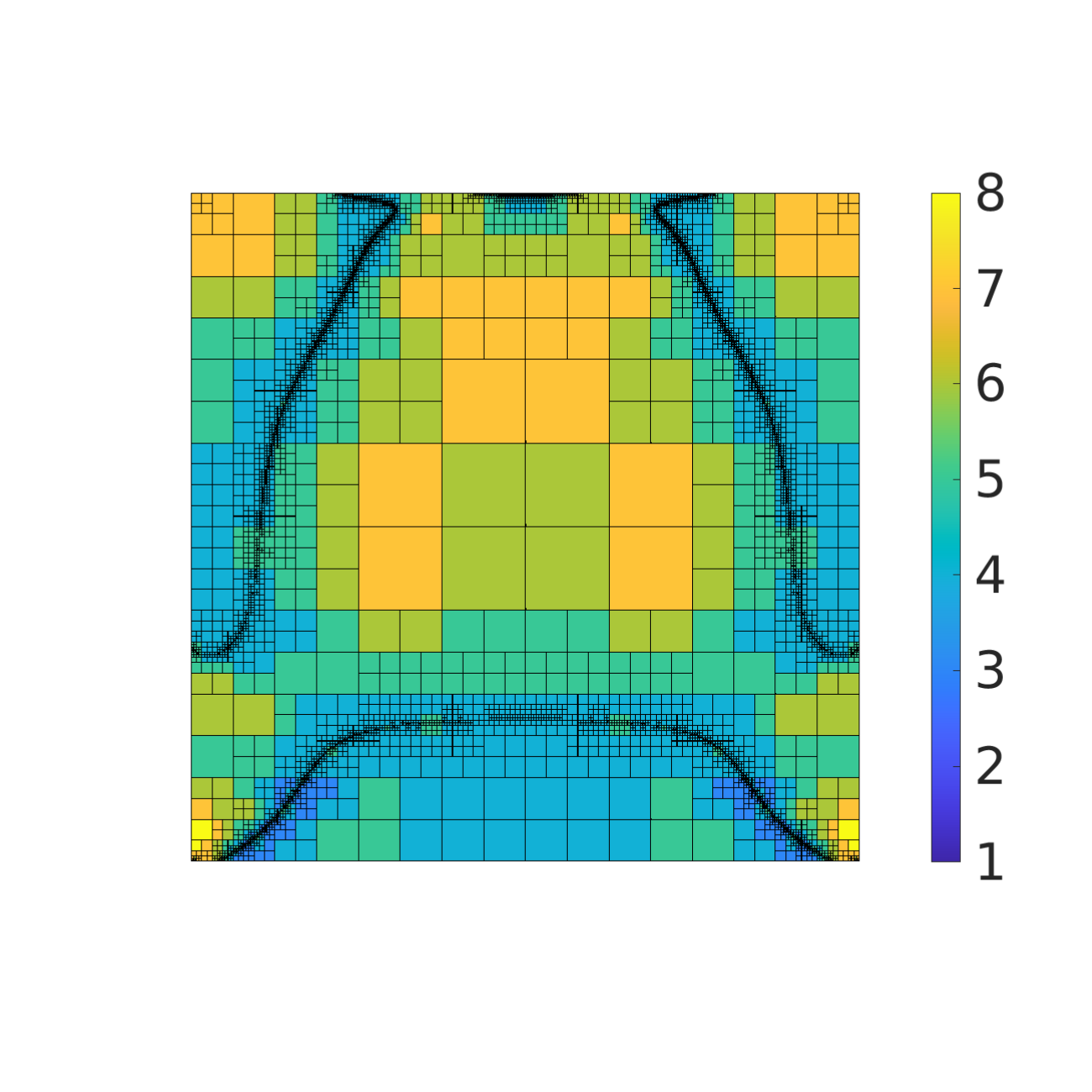}
    \caption{$hp$-adaptive mesh (nr.~35, $936.652$ DOF)}
  \end{subfigure}
  \hspace{0.5cm}
  \caption{\it Adaptively generated meshes with comparable DOF.}
  \label{fig:meshes_2d}
\end{figure}

% ----------------------------------------------------------------------------------------

%\medskip

% ---------------------------------------------------------------------------------------
%   SECTION: Appendix
% ----------------------------------------------------------------------------------------

\section*{Appendix}

Let $X$ be a Hilbert space with associated inner product $(\cdot,\cdot)_X$ and induced norm $\Vert\cdot\Vert_X$. Furthermore, let $\Lambda\subseteq X$ be a non-empty, closed and convex subset. For given $\lambda,p\in X$ consider the minimization problem: Find a $\mu^*\in \Lambda$ such that
\begin{align}\label{eq:abstract_minP}
 E(\mu^*) = \min_{\mu\in \Lambda} E(\mu),
\end{align}
where $E(\cdot)$ is defined as
\begin{align*}
 E(\mu)
 := \Vert \mu - \lambda \Vert_X^2 - (\mu, p)_X.
\end{align*}
Let $\mathcal{P}:X\rightarrow \Lambda$ be the projection operator onto $\Lambda$ with respect to $\Vert\cdot\Vert_X$, i.e.
\begin{align*}
 \Vert \mathcal{P}(\nu)-\nu\Vert_X = \min_{\mu\in\Lambda} \Vert \mu-\nu\Vert_X\qquad 
 \forall \, \nu\in\Lambda.
\end{align*}
It is well-known from the Hilbert projection thereom that $\mathcal{P}(\cdot)$ is well-defined. Furthermore, $\mathcal{P}(\nu)$, for $\nu\in X$, is characterized by the condition
\begin{align}\label{eq:projectionProp_cond}
 \big( \mathcal{P}(\nu) - \mu, \nu - \mathcal{P}(\nu) \big)_X 
 \geq 0 \qquad 
 \forall \, \mu\in X.
\end{align}

\begin{lemma}\label{prop:abstract_minP}
The minimization problem \eqref{eq:abstract_minP} has the unique solution  $\mu^*:= \mathcal{P}(\lambda + \frac{1}{2} \, p)$.
\end{lemma}

\begin{proof}
Defining $\widehat{\mu}:=\lambda + \frac{1}{2} \, p$ and exploiting \eqref{eq:projectionProp_cond} we determine for $\mu\in\Lambda$ with $\mu\neq \mu^*$
\begin{align*}
 E(\mu) 
 &= \Vert \mu \Vert_X^2 - 2 \, (\mu,\lambda)_X + \Vert \lambda \Vert_X^2 - ( \mu, p)_X\\
 &= \Vert \mu \Vert_X^2 - 2 \, (\mu,\widehat{\mu})_X + \Vert \lambda \Vert_X^2\\
 &= \Vert \mu \Vert_X^2 - 2 \, (\mu,\mu^*)_X - 2 \, (\mu, \widehat{\mu} - \mu^*)_X + \Vert \lambda \Vert_X^2\\
 &= \Vert \mu - \mu^* \Vert_X^2 + \Vert \lambda \Vert_X^2 - \Vert \mu^*\Vert_X^2 - 2 \, (\mu,\widehat{\mu} - \mu^*)_X \\
 &= \Vert \mu - \mu^* \Vert_X^2 + \Vert \lambda \Vert_X^2 - \Vert \mu^*\Vert_X^2 + 2 \, (\mu^* - \mu,\widehat{\mu} - \mu^*)_X - 2 \, (\mu^*, \widehat{\mu} - \mu^*)_X\\
 &> \Vert \lambda \Vert_X^2 - \Vert \mu^*\Vert_X^2 - 2 \, (\mu^*, \widehat{\mu} - \mu^*)_X\\
 & = \Vert \lambda \Vert_X^2 - \Vert \mu^*\Vert_X^2 - 2 \, (\mu^*, \lambda)_X - (\mu^*,p)_X + 2\, \Vert \mu^*\Vert_X^2 \\
 &= \Vert \mu^* - \lambda \Vert_X^2 - (\mu^*,p)_X\\
 &= E(\mu^*),
\end{align*}
which shows the assertion.
\end{proof}

% -----------------------------------------------------------------------------------------
%    Acknowledgment, Funding & Bibliography
% -----------------------------------------------------------------------------------------

% \pagebreak

\bibliographystyle{amsalpha}
%\bibliography{...}

\end{document}